\documentclass[english,12pt,a4paper,reqno]{amsart}
\usepackage{amsmath,amsthm,amsfonts,amssymb,amscd}
\usepackage{graphics}

\usepackage{psfrag}
\usepackage{a4wide}
\usepackage[latin1]{inputenc}
\usepackage{subfig}
\usepackage{graphicx}
\usepackage{sidecap}
\usepackage{fancybox}

\usepackage{hyperref}

\def\constant{\operatorname{constant}}

\numberwithin{equation}{section}

\newtheorem{maintheorem}{Theorem}

\newcommand{\cmt}{\begin{maintheorem}}
\newcommand{\fmt}{\end{maintheorem}}

\newtheorem{maincorollary}[maintheorem]{Corollary}

\newcommand{\cmc}{\begin{maincorollary}}
\newcommand{\fmc}{\end{maincorollary}}

\newtheorem{theorem}{Theorem}[section]

\newtheorem{proposition}[theorem]{Proposition}

\newtheorem{lemma}[theorem]{Lemma}
\theoremstyle{remark}
\newtheorem{definition}[theorem]{Definition}
\newtheorem{remark}[theorem]{Remark}

\usepackage{mathrsfs}
\usepackage{nicefrac}
\usepackage{xfrac}

\title[Statistical instability]{Statistical instability for contracting Lorenz flows}
\author{Jos\'e F. Alves}
\address{Jos\'e F. Alves\\ Departamento de Matem\'atica, Faculdade de Ci\^encias da Universidade do Porto\\
Rua do Campo Alegre 687, 4169-007 Porto, Portugal}
\email{jfalves@fc.up.pt} \urladdr{http://www.fc.up.pt/cmup/jfalves}

\author{Muhammad Ali Khan}
\address{Muhammad Ali Khan\\ Departamento de Matem\'atica, Faculdade de Ci\^encias da Universidade do Porto\\
Rua do Campo Alegre 687, 4169-007 Porto, Portugal}
\email{malikhan09@gmail.com}

\date{\today}

\thanks{The authors were partially supported by CMUP (UID/MAT/00144/2013) and the project PTDC/MAT-CAL/3884/2014 funded by Funda\c{c}\~ao para a Ci\^encia e a Tecnologia (FCT) Portugal with national (MEC) and European structural funds through the program FEDER, under the partnership agreement PT2020. JFA was also  supported by The Leverhulme Trust VP2-2017-004 Visiting Professorship and MAK by  the FCT  grant SFRH/BD/93856/2013.}



\begin{document}

\begin{abstract}
We consider one parameter families of vector fields introduced by Rovella, obtained through modifying the eigenvalues of the geometric Lorenz attractor, replacing the expanding condition on the eigenvalues of the singularity by a contracting one. 
We show that there is no statistical stability within the set of parameters for which there is a physical measure supported on the attractor. This is achieved obtaining a similar conclusion at the level of the corresponding one-dimensional contracting Lorenz maps.
\end{abstract}

\maketitle

\tableofcontents

\section{Introduction}

It is a fundamental problem in Dynamics to understand under which conditions the behavior of typical (positive Lebesgue measure) orbits is well defined from the statistical point of view and under which conditions these statistical properties are stable under small modifications. In uniformly hyperbolic  dynamics, the statistical properties of a dynamical system can be expressed through \emph{Sinai-Ruelle-Bowen (SRB) measures},   introduced by Sinai for Anosov diffeomorphisms \cite{Si1972} and  obtained by Ruelle and Bowen for Axiom A attractors, both for diffeomorphisms \cite{Ru1976} and flows \cite{BoRu1975}. These measures are characterised by having at least one positive Lyapunov exponent almost everywhere and conditional measures on local unstable manifolds which are absolutely continuous with respect to the conditional Lebesgue measure on those manifolds. In many situations,  including all the classical systems studied by Sinai, Ruelle and Bowen, SRB measures are  a   particular case of physical measures that we introduce next.


\subsection{Statistical instability} 
%
We say that a Borel probability measure $\mu$   invariant by a flow $(X^{t})_{t}$ for a vector field $X$ in Riemannian manifold~$M$ is a \emph{physical measure} for $X$ if there is a positive Lebesgue measure subset of points $x\in M$    such that
$$
\lim\limits_{T\to + \infty} \frac{1}{T} \int_{0}^{T} \varphi (X^{t}(x))dt = \int \varphi \mbox{ } d\mu, \quad \mbox{ for any continuous $\varphi : M \rightarrow \mathbb{R}$. }
$$
Physical measures for discrete-time dynamical systems are defined similarly, replacing the continuous time averages by the corresponding discrete time averages in the formula above. 
A special type of physical measure arises when  we have  an attracting periodic orbit. Clearly, the singular measure supported on that periodic orbit is a physical measure. 
The aforementioned SRB measures for hyperbolic attractors appear more generally  in the setting chaotic attractors, where there exist directions of expansion within  the attractor. 

While studying the persistence of the statistical properties of Viana maps, the notion of \emph{statistical stability} for certain families of dynamical systems has been proposed in \cite{AlVi2002}, trying to express the continuous variation of the physical measure as a function of the dynamical system.
%
This kind of stability essentially states that small perturbations of the system do not cause much effect on the averages of continuous observables along   orbits. 
Besides the aforementioned statistical stability for Viana maps,  in the recent years  several other results have been obtained for families of chaotic maps, including unimodal maps \cite{BS09, BS12, Fr2005,Fr2010, RS92,T96}, H\'enon diffeomorphisms \cite{ACF1,ACF2,Ur1995,Ur1996} and Lorenz-like maps or flows \cite{AlSo2012, AlSo2014, BR18}. 

Here we are interested in results in the opposite direction. We say that a parametrised family of vector fields   $(X_a)_{a\in \mathcal P}$  (or the corresponding family of flows)  is \emph{statistically unstable} at a certain parameter $a\in \mathcal P$ if there is a sequence $(a_n)_n$ in $\mathcal P$ converging to $a$ such that
   each  $X_{a_n}$ has a physical measure $\mu_{a_n}$
 and, moreover,  the sequence $(\mu_{a_n})_n$ does not converge (in the weak* topology) to a physical measure of $X_a$. Statistically unstable families of discrete-time dynamical systems are defined similarly.

There are not many examples of statistically unstable  systems in the literature. For results in this direction, see   \cite{HoKe1990,Th2001} for  the quadratic family or \cite{Ke1982} for piecewise expanding maps, both  discrete time dynamical systems. In this work, we show that the family of contracting Lorenz flows introduced by Rovella \cite{Ro1993}  and the associated family  one-dimensional maps are both statistically unstable. To  the best of our knowledge, this gives  the first example of a  statistically unstable family of vector fields.

\subsection{Contracting Lorenz  flows}
Lorenz \cite{Lo1963} formulated a simple model of differential equations in $\mathbb{R}^{3}$ as a finite dimensional approximation of the evolution equation for atmospheric dynamics, numerically showing the existence of an attractor with sensitive dependence on initial conditions. It was then a question of great interest to rigorously prove this experimental evidence. Motivated by this problem, Guckenheimer and Williams \cite{GuWi1979} tried to write down the abstract properties of that attractor and produced a prototype, the so-called \emph{geometric Lorenz attractor}, which turned out to be the first example of a robust chaotic attractor with a hyperbolic singularity. Given as the 14th problem of Smale \cite{Sm1998}, the question of knowing if the dynamics of the Lorenz equations is same as that of the geometric model. 
This problem had a positive answer by Tucker~\cite{Tu1999}. 

The geometric Lorenz attractor is a maximal invariant set for a vector field $X$ in $\mathbb{R}^{3}$ having a dense orbit with a positive Lyapunov exponent and a singularity at the origin, whose derivative has real eigenvalues satisfying
$$
0<-\lambda_{3}<\lambda_{1}<-\lambda_{2}.
$$
The contracting Lorenz attractor, introduced by Rovella in \cite{Ro1993}, is the maximal invariant set of a vector field  whose construction is similar to geometric Lorenz attractor, with the only difference that the eigenvalues for the derivative at the singularity satisfy 
$$
0<\lambda_{1}<-\lambda_{3}<-\lambda_{2}.
$$
This attractor is no longer topologically robust. Only in a measure theoretical sense one can detect some robustness:  there is a codimension two submanifold in the space of all vector fields, whose elements are full density points for the set of vector fields that exhibit a contracting Lorenz attractor in generic two parameter families through them. Rovella observed  that it is enough to consider one parameter families of vector fields in that codimension two submanifold, and showed that for any such family $(X_{a})_{a\ge0}$ there is a positive Lebesgue measure subset of parameters $\mathcal R\subset \mathbb{R}^{+}$  such that the vector field $X_{a}$ has a chaotic attractor for each $a\in\mathcal{R}$. We will refer to  the flow of each $ X_{a} $   as a   \emph{contracting Lorenz flow} and to $\mathcal R$ as  the set of \emph{Rovella parameters}.

 Metzger managed to prove in~\cite{Me2000} that  the strange attractor corresponding to a Rovella parameter supports a unique physical measure, which is in fact an SRB measure. In~\cite{Me-Stochastic2000}, Metzeger proved   the stability of this measure under random perturbations (\emph{stochastic stability}).
Our first main result gives  that from a deterministic point of view the   situation is completely different.   

\begin{maintheorem} \label{thmainA}
Given any $a\in\mathcal{R}$, there is a sequence $(a_{n})_{n}$ in $\mathbb R^{+}$ converging to $a$ such that for each $a_{n}$ the  Dirac measure supported on the singularity  contained in the attractor of~$X_{a_n}$ is a physical measure for the flow of $X_{a_{n}}$.
\end{maintheorem}

Recalling that by~\cite{Me2000}   each  Rovella parameter  has a   unique physical measure supported on the strange attractor, which is actually  an SRB measure, from Theorem~\ref{thmainA}  we easily get the following important  consequence.

\begin{maincorollary} \label{comain}
Contracting Lorenz flows  are statistically unstable at   Rovella parameters.
\end{maincorollary}

This shows that for the  families of contracting Lorenz flows    considered by Rovella, the situation  is  completely different from the classical Lorenz  flows, where statistical stability holds everywhere; see  \cite{AlSo2014, BR18}. 

It is worth noting that Rovella established in  \cite{Ro1993} that parameters with chaotic attractors are accumulated by others with attracting periodic orbits. However, no conclusion has been drawn about the convergence (or not) of the physical measures supported on these attracting periodic orbits to the SRB measure supported on the chaotic attractor for the limiting parameter. Note also that the  physical measures corresponding to our  sequence of parameters in Theorem~\ref{thmainA} are of a different nature: they are supported on a singularity which has one positive eigenvalue, clearly not an attracting singular orbit. 

\subsection{One-dimensional contracting Lorenz maps}
The proof of Theorem~\ref{thmainA} uses the  key  fact that, as in the classical situation, contracting Lorenz  flows have a global cross-section with a one dimensional invariant foliation which is contracted by the first return map; see  \cite{Ro1993}. Quotienting by stable leaves we get  a one parameter family $\{ f_{a} \}_{a \ge 0}$ of one-dimensional maps, which we shall refer to as the family of  \emph{contracting Lorenz maps}. Each $f_{a}$ carries a discontinuity at~$0$ and two critical values $\pm 1$; see Subsection~\ref{sec1.1.3} for details. 
Using  the strategy of Benedicks and Carleson \cite{BeCa1985, BeCa1991} for the quadratic family, 
Rovella shows in \cite{Ro1993}  that the critical values $\pm1$ of $f_{a}$ have positive Lyapunov exponents, thus obtaining a strange attractor for each $X_{a}$ with $a\in\mathcal{R}$. Metzeger~\cite{Me2000} showed that each one-dimensional map $f_{a}$ with $a\in\mathcal{R}$ has a unique physical measure, which is in fact absolutely continuous invariant probability measure. This yields an SRB measure supported on the attractor of $X_{a}$. 
%

Here we will also use the family of contracting Lorenz maps to prove Theorem~\ref{thmainA}.
Inspired by the work of Thunberg~\cite{Th2001} for the quadratic family, we will obtain parameters with a \emph{super-attractor}, i.e. an attracting  periodic orbit containing the critical point, accumulating on Rovella parameters. To each of the parameters in the sequence   given by Theorem~\ref{thmain} corresponds a flow for which the unstable manifold of the singularity in  the attractor is contained in its stable manifold. 


\begin{maintheorem} \label{thmain}
Given any $a\in\mathcal{R}$, there is  a sequence  $(a_{n})_{n}$ in $\mathbb{R}^{+}$ converging to $a$ such that each~$f_{a_{n}}$ has a super-attractor. Moreover, the sequence of physical measures supported on these super-attractors converges to an invariant measure for~$f_{a}$ supported on a repelling periodic orbit.
\end{maintheorem}

From Theorem~\ref{thmain} we can deduce the following interesting conclusion:

\begin{maincorollary} \label{comainf}
Contracting Lorenz maps  are statistically unstable at   Rovella parameters.
\end{maincorollary}

It is enough to see that, for $a\in\mathcal R$, a physical measure for $f_a$ cannot be supported on a repelling periodic orbit.
In fact,  each  $f_a$ with $a\in\mathcal{R}$  has a dense orbit, by \cite[Theorem 2]{Ro1993}. Also,   for every nonuniformly expanding map, forward invariant sets with positive Lebesgue measure must have full Lebesgue measure in some interval of a fixed radius (not depending on that set), by \cite[Lemma~5.6]{AlBoVi2000}. 
So, applying this fact to the basins of two possible physical measures, together with the existence of dense orbits, we easily see that there is at least one common  point in the basins of both physical  measures, and so they coincide. Since~\cite{Me2000} gives that each $f_{a}$ with $a\in\mathcal{R}$ has a  physical measure which is absolutely continuous with respect to Lebesgue measure, it follows that the measure supported on a repelling periodic orbit cannot be a physical measure for $f_a$.

Notice that, for each $a\in\mathcal R$,  the invariant measure for~$f_{a}$  given by Theorem~\ref{thmain} lifts to a measure
supported on a periodic orbit (of saddle type) in the Poincaré
section, and this measure lifts  to a measure supported on a periodic orbit for the corresponding $X_a$.
Since projections (both from the ambient manifold to the Poincaré section, and from the Poincaré section to the quotient interval) preserve physical measures, it easily follows that the measure supported on the periodic orbit for $X_a$ cannot be a physical measure.

In the opposite direction, using techniques developed in \cite{Al2004,Fr2005,Fr2010}, Alves and Soufi \cite{AlSo2012}  obtained the \emph{strong statistical stability} for Rovella maps within the set~$\mathcal{R}$: the density of the physical measure (which is absolutely continuous with respect to Lebesgue measure on the interval) depends continuously (in the $L^{1}$-norm) on the parameter $a\in\mathcal{R}$. 
The weak* continuity of the physical measures for the flows within the set of   Rovella parameters is the goal of the work in progress~\cite{AlSo2019}. 

\medskip
\noindent\textbf{Acknowledgement}. The authors acknowledge  interesting discussions with Stefano Luzzatto at Tarbiat Modares University, Tehran, that much contributed to the final statement of Theorem~\ref{thmainA}. 


\section{Lorenz-like attractors} \label{sec1.1}
Let $M$ be a manifold and $X$ be a smooth vector field on $M$ and denote by $X^{t}$ the flow generated by $X$. An \emph{attractor} for $X^{t}$ is a transitive (it contains a dense orbit) invariant set $\Lambda \subset M$ such that it has an open neighborhood $U$ with $X^{t}\left(\overline{U}\right) \subset U$ for all $t>0$ and
$$
\Lambda = \bigcap_{t \ge 0}X^{t}(U).
$$
A set $U$ with these properties is called a \emph{trapping region} for the attractor $\Lambda$. We say that $\Lambda$ is \emph{robust} if for any smooth vector field $Y$ in a neighborhood of $X$, we still have ${\cap_{t \ge 0}}Y^{t}(U)$ an attractor. \par

\subsection{Geometric Lorenz attractor} \label{sec1.1.1}\label{sec1.1.2}     
Lorenz \cite{Lo1963} studied numerically the vector field $X$ given by the system of differential equations in $\mathbb R^3$
$$
\left \{
\begin{array}{l}      
    \dot{x} = a(y-x)\\
    \dot{y} = bx-y-xz\\
    \dot{z} = xy-cx
\end{array}\right.
$$
for the parametric values $a=10$, $b=28$ and $c=8 / 3$. 
The following properties are well known for this vector field:
\begin{enumerate}
\item  $X$ has a singularity at the origin with eigenvalues
$$
 0< 2.6 \approx  -\lambda_{3} < \lambda_{1} \approx 11.83 < -\lambda_{2}  \approx 22.83;
$$
\item  there is a trapping region $U$ such that $\Lambda = \bigcap_{t>0} X^{t}(U)$ is an attractor and the origin is the unique singularity contained in $U$;
\item  $\Lambda$ contains a dense orbit with a positive Lyapunov exponent.
\end{enumerate}
A set $\Lambda$ with the above properties is usually referred as a \emph{strange attractor}. 

In the late 1970's, Guckenheimer and Williams \cite{GuWi1979} introduced the geometric description of a flow having similar dynamical behavior as that of Lorenz system, known as \emph{geometric Lorenz flow}. This geometric model posses a trapping region containing a transitive attractor which has a singularity accumulated by the regular orbits preventing the attractor to be hyperbolic. 

The construction of the geometric model can be briefly described as follows: the vector field $X$ has a singularity at $(0,0,0)$ and it is linear in a neighborhood containing the cube $\{(x,y,z):|x|\le 1, |y| \le 1, |z| \le 1 \}$. The derivative of $X$ at the singularity admits three real eigenvalues $\lambda_{1}$, $\lambda_{2}$ and $\lambda_{3}$ satisfying $0<-\lambda_{3}<\lambda_{1}<-\lambda_{2}$. This means that the origin is a saddle point with a 2-dimensional stable manifold. We denote by $\Sigma$ the roof $\{|x| \le 1, |y| \le 1, z = 1 \}$ of the cube, intersecting the stable manifold of the singularity along a curve $\Gamma$ which divides $\Sigma$ into two regions $\Sigma^{+}=\{(x,y,1) \in \Sigma: x > 0 \}$ and $\Sigma^{-}=\{(x,y,1) \in \Sigma: x < 0 \}$.
The images of the rectangles $\Sigma^{\pm}$, by the return map, are triangles $S^{\pm}$ except vertices $(\pm 1, 0,0)$ such that the line segments $\{x=constant\} \cap \Sigma$ are mapped to the segments $\{z=constant\} \cap S^{\pm}$. Then we assume that the line segments $\{z=\constant\} \cap S^{\pm}$ are mapped to the segments contained in $\{x=\constant\} \cap \Sigma$. Consequently, we obtain the following expression for Poincar\'{e} return map
$$
P(x,y)=(f(x),g(x,y)),
$$
for some maps $f: I \setminus \{0 \} \rightarrow I$ and $g:  I \setminus \{0 \} \times  I \rightarrow I$, with $I=[-1,1]$. The one dimensional map $f$ is shown in Figure \ref{Lmap} and has the following properties:
\begin{enumerate}
\item  $\lim\limits_{x \to 0^{+}} f(x)=-1$ and $\lim\limits_{x \to 0^{-}} f(x)=1$;
\item  $f$ is differentiable on $I \setminus \{0 \}$ and $f^{\prime}(x)>\sqrt{2}$ for all $x \in I \setminus \{0 \}$;
\item  $\lim\limits_{x \to 0^{+}} f^{\prime}(x)=\lim\limits_{x \to 0^{-}} f^{\prime}(x)=+ \infty$.
\end{enumerate}
Moreover, there exists a constant $\rho<1$ such that $|\frac{\partial g}{\partial y}|<\rho$. This implies that the foliation given by the segments $\Sigma \cap \{x=\constant\}$ contracts uniformly: there exists a constant $C>0$ such that for any leaf $\gamma$ of the foliation, $p,q \in \gamma$ and  $n \in \mathbb{N}$, we have
$$
\mathrm{dist}(P^{n}(p), P^{n}(q)) \le C \rho^{n} \mathrm{dist}(p,q).
$$

\begin{figure}
\centering
\includegraphics[width=0.5\textwidth]{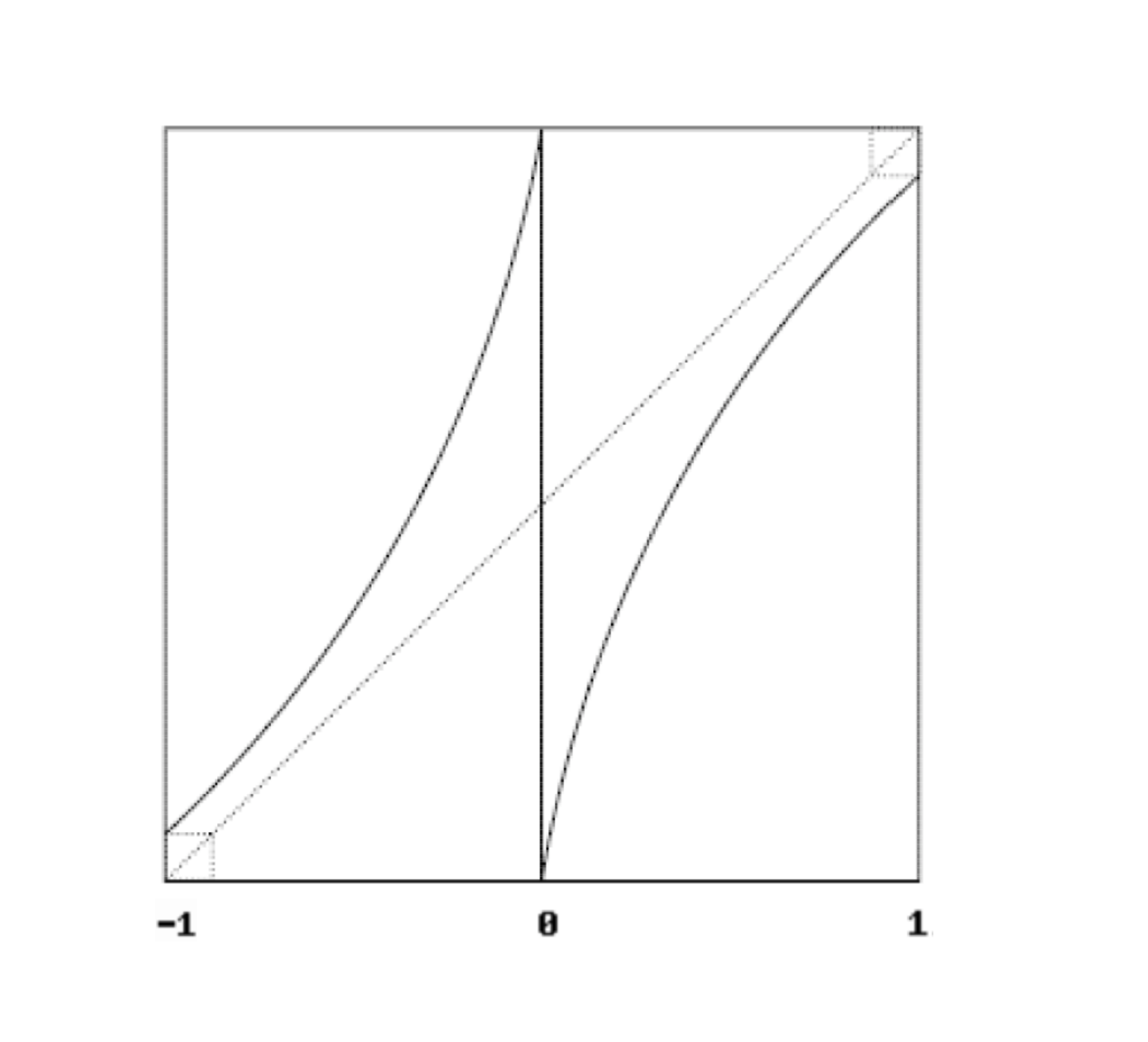}
$\vspace{-10pt}$ 
\caption{Lorenz map}\label{Lmap}
\end{figure}
%
An important fact about the geometric Lorenz attractor is robustness: vector fields $C^{1}$-close to the one constructed above also admit strange attractors. 
%
Note that $X$ has a hyperbolic singularity and the cross section $\Sigma$ is transversal to any flow $C^{1}$-close to $X$. Therefore the singularity persists and the eigenvalues satisfy the same relations for every vector field $Y$ in a $C^{1}$-neighborhood $\mathcal U$ of $X$. Moreover through a $C^{1}$ change of coordinates, the singularity of any $Y \in \mathcal{U}$ stands on the origin and the derivative of $Y$ at origin has eigenvectors in the direction of coordinate axis as before, whereas the stable manifold of the singularity remains the plane $x=0$. Consequently, $Y$ has a Poincar\'{e} return map and a 1-dimensional quotient map $f_Y$ with properties similar to $P$ and $f$, respectively.
%
%

\subsection{Contracting Lorenz attractor} \label{sec1.1.3} 
Considering a vector field similar to that used by Guckenheimer and Williams \cite{GuWi1979}, Rovella \cite{Ro1993} introduced a different kind of attractor $\Lambda$ named as \emph{contracting Lorenz attractor}. 
The flow associated to this attractor has similar construction as that of geometric one with the initial vector field $X_{0}$ in $\mathbb{R}^{3}$ which has the following properties:
\begin{enumerate}
\item  $X_{0}$ has a singularity at the origin whose derivative has three real eigenvalues $\lambda_{1},\lambda_{2}$ and $\lambda_{3}$ satisfying:
\begin{enumerate}
\item  $0<\lambda_{1}<-\lambda_{3}<-\lambda_{2}$,
\item  $r>s+3$, where $r= \frac{-\lambda_{2}}{\lambda_{1}}$ and $s=\frac{-\lambda_{3}}{\lambda_{1}}$;
\end{enumerate}
\item  There exists an open set $U \in \mathbb{R}^{3}$ forward invariant by the flow and containing the cube $\{ (x,y,z): |x| \le 1, |y| \le 1, |z| \le 1 \}$. The top of the cube $\Sigma$ is foliated by stable line segments $\{ x=\constant \} \cap \Sigma$ which are invariant by the Poincar\'{e} return map $P_{0}$. This gives rise to a one dimensional map $f_{0}: I \setminus \{ 0 \} \rightarrow  I$ such that
$$
f_{0} \circ \pi = \pi \circ P_{0},
$$
where $\pi$ is the canonical projection along stable leaves;
\item The stable leaves $x=constant$ in $\Sigma$ are uniformly contracted by the Poincar\'e map.
\end{enumerate}


The main idea adopted by Rovella was to replace the expanding condition $\lambda_{1}+\lambda_{3}>0$ of the geometric flow by the contracting condition $\lambda_{1}+\lambda_{3}<0$.  

There are still some properties of the initial vector field $X_{0}$ which are valid for the $C^{3}$ perturbations. Consider a small neighborhood $\mathcal{U}$ of $X_{0}$ such that each $X \in \mathcal{U}$ has a singularity near the origin with eigenvalues $\lambda_{1}(X), \lambda_{2}(X), \lambda_{3}(X)$ satisfying $-\lambda_{2}(X)>-\lambda_{3}(X)>\lambda_{1}(X)>0$ and $r_{X}>s_{X}+3$, where $r_{X}=-{\lambda_{2}(X)}/{\lambda_{1}(X)}$ and $s_{X}=-{\lambda_{3}(X)}/{\lambda_{1}(X)}$. Moreover, the trajectories contained in the stable manifold of the singularity still intersect $\Sigma$. The set $\mathcal{U}$ can be taken small enough so that the trapping region $U$ is still forward invariant under the flow of every $X \in \mathcal {U}$. The existence of $C^{3}$ $1$-dimensional stable foliations in $U$ and their continuous variation with $X$ was proved by Rovella in \cite{Ro1993}.

For each $X \in \mathcal{U}$, we may take a square $\Sigma_{X}$ close to $\Sigma$ formed by line segments of the foliations so that the first return map $P_{X}$ to $\Sigma_{X}$ has an invariant foliation and we can choose the coordinates $(x,y)$ in $\Sigma_{X}$ so that the segment $x=0$ corresponds to the stable manifold of the singularity and
$
P_{X}(x,y)=(f_{X}(x),g_{X}(x,y)).
$
The map $f_{X}$ is of class $C^{3}$ everywhere but at $x=0$ where it has a discontinuity.\par

In order to prove his main result, Rovella considered a one parameter family $\{ X_{a} \in \mathcal{U}:a \ge 0 \}$ of vector fields and the corresponding family $\{ f_{a}:I\setminus \{ 0 \} \rightarrow I: a \ge 0 \}$ of $C^{3}$ one dimensional maps as shown in Figure~\ref{Rmap}, with the following properties:
\begin{enumerate}
%
\item [(A0)]  $f_{0}(1)=1$ and $f_{0}(-1)=-1$;
\item [(A1)] $f_{a}(0^{+})=-1$ and $f_{a}(0^{-})=1$;
\item [(A2)] $f_{a}'>0$,  $f_{a}''\vert_{[-1,0)}<0$ and $f_{a}''\vert_{(0,1]}>0$; 
\item [(A3)] there exist $K_{0},K_{1}>0$ and $s>1$ (independent of $a$) such that for all $x\in I\setminus \{ 0 \}$
$$
K_{0}|x|^{s-1} \le f_{a}^{\prime}(x) \le K_{1}|x|^{s-1};
$$

\item [(A4)] $f_{a}$ has negative Schwarzian derivative:  there is $\chi<0$ such that for all $x\in I\setminus \{ 0 \}$
$$
S(f_{a})(x)=\left(\frac{f_{a}^{\prime \prime}}{f_{a}^{\prime}}\right)^{\prime}(x)-\frac{1}{2} \left(\frac{f_{a}^{\prime \prime}}{f_{a}^{\prime}}\right)^{2}(x)<\chi;
$$
\item [(A5)] $f_{a}$ depends continuously on $a$ in the $C^{3}$ topology;
\item [(A6)] the functions $a \rightarrow f_{a}(\pm 1)$ have derivative $1$ at $a=0$. 
\end{enumerate}

\begin{figure} 
\centering
\includegraphics[width=0.45\textwidth]{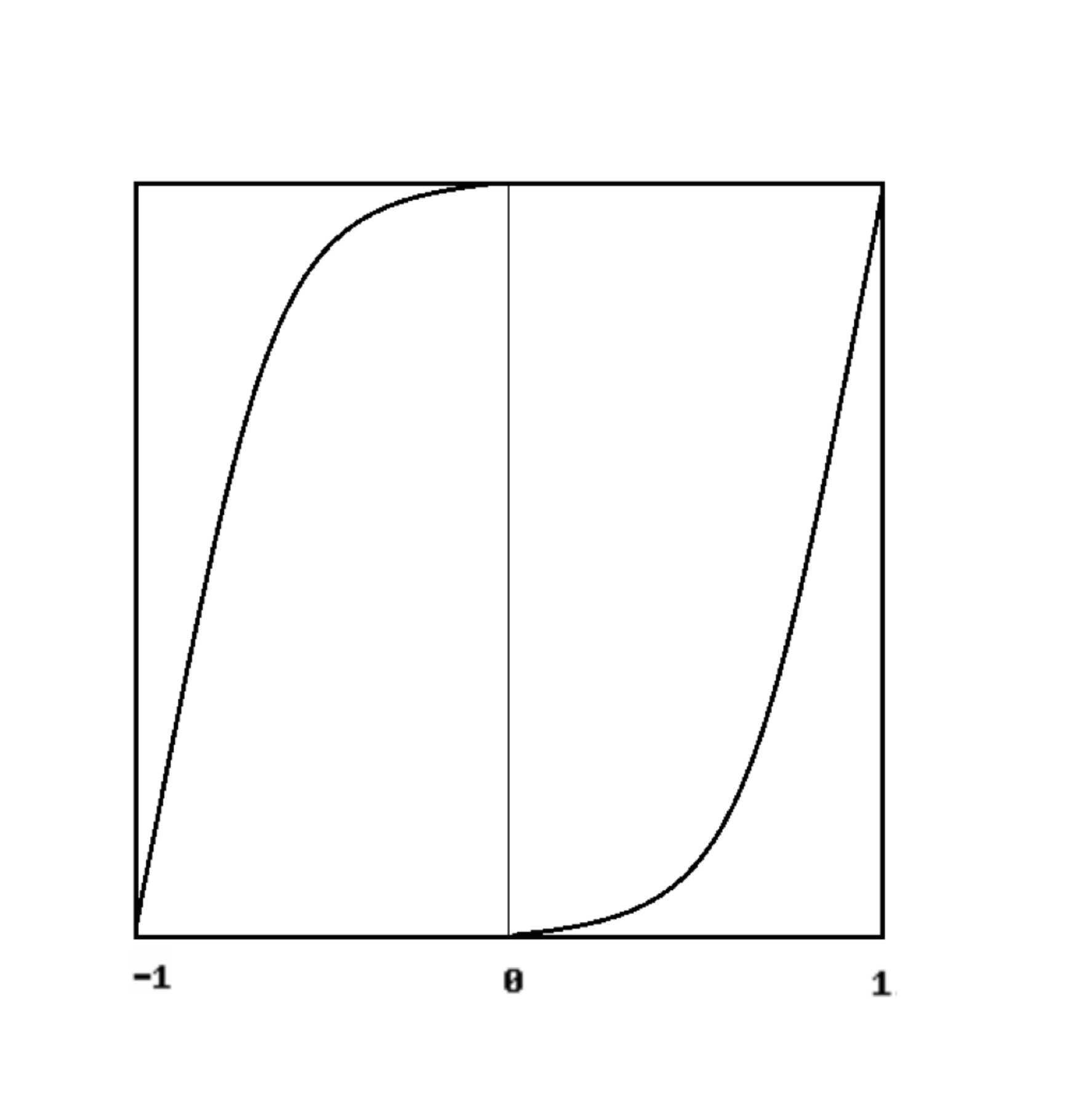}
$\vspace{-20pt}$ 
\caption{Contracting Lorenz map} \label{Rmap}
\end{figure}

Comparing to the one-dimensional family of maps associated to the classical geometric Lorenz attractor, the big difference lies on the fact that the discontinuity point has no longer infinite side derivatives, but zero side derivatives. In particular, these maps are not piecewise expanding. 
For definiteness, we assume that $f_{a}(0)=-1$ for every $a\ge0$. This corresponds to extending each map $f_{a}$ to the critical point $0$ continuously on the right hand side, and enables us to consider the family of dynamical systems $f_{a}:I  \rightarrow I$, for $a \ge 0$.

\section{Statistical instability for the flows} \label{sec5}
In this section we prove  Theorem~\ref{thmainA}, assuming that  Theorem~\ref{thmain} holds.  Consider the family of vector fields $(X_{a})_{a\ge0}$ and the family of one-dimensional maps $f_{a}:I\to I$ as before. Recall that we are assuming that $f_{a}(0)=-1$ for every $a\ge0$. Coherently, we extend the Poincar\'e map $P_{a}:\Sigma\to\Sigma$ to the critical line $\{x=0\}$ continuously on the right hand side. Observe that the image of this critical line is a single point in $\{x=-1\}$.

Given a parameter $a\in\mathcal{R}$, let $(a_{n})_{n}$ be a sequence of parameters converging to $a$ as in Theorem~\ref{thmain}. For each $n$, consider $\{z_{1},\dots,z_{k}\}$ the super-attractor of $f_{a_{n}}$, i.e. the attracting periodic orbit (of period $k$) containing the critical point $0$. Using the fact that the stable foliation is contracted uniformly, we easily deduce that there is an attracting periodic orbit $\{Z_{1},\dots,Z_{k}\}$  for $P_{a_{n}}$ as well. As this attracting periodic orbit contains an iterate in the discontinuity region of the Poincar\'e map we cannot ensure that its topological basin contains a neighbourhood of itself, but at least it contains some open set $B\subset\Sigma$. Assume that $P_{a_{n}}(Z_{i})=Z_{i+1}$ for each $1\le i\le k-1$ and $Z_{k}$ is the point in the periodic orbit that belongs to the critical line $\{x=0\}$. 

Let us now prove that the Dirac measure $\delta_{0}$ on the singularity $0$ of the vector field $X_{a_{n}}$ is a physical measure. 
Consider any continuous function $\varphi:U\to \mathbb{R}$. Given an arbitrary $\epsilon>0$, let $A$ be a small neighbourhood of $0$ such that 
$$
|\varphi(x)-\varphi(0)|<\varepsilon,\quad\text{for all $x\in A$.}
$$
Given any point $x\in B$, we may find $L>0$ such that the time spent by the orbit of $x$ between two consecutive visits to $A$ is at most $L$. On the other hand, as $X_{a_{n}}^t(Z_{k})\to 0$ when $t\to\infty$, denoting by $T_{1},T_{2},\dots$ the consecutive periods of time the orbit of $x$ spends in $A$ at each visit, we have that $T_{m}\to\infty$ as $m\to\infty$. 
Hence, given  $T>0$, we may consider moments $0=s_{0}<t_{0}<s_{1}<t_{1}<\dots< s_{m}< t_{m} \le s_{m+1} =T$  such  that 
for each $i= 1,\dots,m$, we have
\begin{enumerate}
\item $X_{a_{n}}^{t}(x)\in A$,\; for all  $t_{i-1}< t\le s_{i}$;
\item $t_{i}-s_{i}\le 2L $;
\item $ s_{i}-t_{i-1}\ge T_i/2$.
\end{enumerate}
Thus, we may write
\begin{eqnarray}
\frac{1}{T}\int_{0}^{T}\varphi\left(X_{a_{n}}^{t}(x)\right) dt&=&\frac{1}{T}\left(\sum_{i=0}^{m } \int_{s_{i}}^{t_{i}}\varphi\left(X_{a_{n}}^{t}(x)\right) dt+\sum_{i=0}^{m } \int_{t_{i}}^{s_{i+1}}\varphi\left(X_{a_{n}}^{t}(x)\right) dt\right)\label{eq.ergcont1}\\
&<&\frac{1}{T} \sum_{i=0}^{m }(t_{i}-s_{i})\|\varphi\|_{0}+\frac{1}{T} \sum_{i=0}^{m }(s_{i+1}-t_{i})\left(\varphi(0)+\varepsilon\right)\nonumber\\
&<&\frac{2Lm}{T}\|\varphi\|_{0}+\varphi(0)+\varepsilon.\nonumber
\end{eqnarray}
Now, using that $T_{m}\to\infty$ as $m\to\infty$ and $m\to\infty$ as $T\to\infty$, we easily get that 
 \begin{equation}\label{eq.Tm}
 \frac{m}{T}\le \frac{2m}{T_{1}+\cdots+T_{m}}\longrightarrow 0,\quad\text{as $T\to\infty$}.
 \end{equation}
Hence
$$
\frac{1}{T}\int_{0}^{T}\varphi\left(X_{a_{n}}^{t}(x)\right)  dt\le \varphi(0)+\varepsilon, \quad\text{for large $T$.}
$$
Using again equality \eqref{eq.ergcont1}, we can also show that
\begin{eqnarray*}
\frac{1}{T}\int_{0}^{T}\varphi\left(X_{a_{n}}^{t}(x)\right)  dt
&\ge&-\frac{1}{T} \sum_{i=0}^{m }(t_{i}-s_{i})\|\varphi\|_{0}+\frac{1}{T} \sum_{i=0}^{m }(s_{i+1}-t_i)\left(\varphi(0)-\varepsilon\right)\\
&\ge &-\frac{1}{T}\|\varphi\|_{0}+\left(1-\frac{1}{T}\sum_{i=0}^{m }(t_{i}-s_{i})\right)\left(\varphi(0)-\varepsilon\right)\\
&\ge &-\frac{1}{T}\|\varphi\|_{0}+\left(1-\frac{2Lm}{T}\right)\left(\varphi(0)-\varepsilon\right).
\end{eqnarray*}
From \eqref{eq.Tm} we get
$$
\frac{1}{T}\int_{0}^{T}\varphi\left(X_{a_{n}}^{t}(x)\right) dt\ge \varphi(0)-\varepsilon, \quad\text{for large $T$.}
$$
Since $\varepsilon>0$ is arbitrary, we have proved that for all $x\in B$
$$
\lim_{T\to\infty}\frac{1}{T}\int_{0}^{T}\varphi\left(X_{a_{n}}^{t}(x)\right) dt = \varphi(0).
$$
As $B$ is a nonempty open subset of $\Sigma$, considering the points whose orbits pass through the points in $B$, we easily get that the basin of $\delta_{0}$ has positive Lebesgue measure in $U$, and so $\delta_{0}$ is a physical measure for the flow of $X_{a_{n}}$, for each parameter $a_{n}$. This gives the conclusion of Theorem~\ref{thmainA}.

\section{The set of Rovella parameters} \label{sec2}
The rest of this paper is devoted to the proof of Theorem~\ref{thmain}.
One of our main goals is to obtain Proposition~\ref{mainlem}, which will be used to show that each Rovella parameter is  accumulated by other parameters  whose critical orbit hits a repelling periodic point. 
To prove it,
we need to explain Rovella's construction of the set $\mathcal{R}\subset \mathbb{R}^{+}$ for the contracting Lorenz family $\{ f_{a}:I \rightarrow I: a \ge 0 \}$ in detail, specially for introducing the notion of escape time in Subsection~\ref{escp}, that has not been addressed in~\cite{Ro1993} and plays a fundamental role in our argument.  

As referred in~\cite{Ro1993}, the construction of $\mathcal{R}$ follows the approach in \cite{BeCa1985, BeCa1991} for the quadratic family.  The basic idea is to construct inductively a nested sequence of parameter sets   $\{ R_{n} \}_{n \in \mathbb{N}}$ such that the derivative of each map associated to $R_{n}$ has \emph{exponential growth} along the two critical values up to time $n$:  there is some $\lambda>1$ 
such that for every $a\in R_{n}$
\begin{align}\label{(EGn)}\tag{EG$_{n}$}
D_{j}^{\pm}(a):=(f_{a}^{j})^{\prime}( \mp1) \ge \lambda^{j}, \quad \mbox{ for } j=1, \ldots, n.
\end{align} 
In addition, those parameters satisfy the so called $\emph{basic assumption}$: for $\alpha>0$ sufficiently small
\begin{align}\label{(BAn)}\tag{BA$_{n}$}
|\xi_{j}^{\pm}(a)| \ge e^{-\alpha j}, \quad \mbox{ for } j=1, \ldots, n,
\end{align} 
where 
 $\xi_{k}^{\pm}(a)= f_{a}^{k-1}(\mp1) $ for all $k\ge 1$.
Condition  \eqref{(BAn)} is imposed to keep $\xi_{n}^{\pm}(a)$ away from the critical point, in particular ensuring that $D_{n}^{\pm}(a)$ do not vanish for a parameter $a$ satisfying (EG$_{n-1}$). 
The key idea is to split the orbit $\{\xi_{k}^{\pm}(a), \mbox{ } k\ge1\}$ into pieces, corresponding to three types of iterates: \emph{returns} $\gamma_{i},$ \emph{bound periods} $\{\gamma_{i}+1, \ldots, \gamma_{i}+p_{i}\}$, and \emph{free periods} $\{\gamma_{i}+p_{i}+1, \ldots, \gamma_{i+1}-1\}$ before the next return $\gamma_{i+1}$. The returns correspond to times at which the orbit visits a small neighborhood of 0; the bound periods consist of times when the orbit, after hitting that small neighborhood, shadows one of the critical orbits closely; the period of times when orbit stays outside that small neighborhood as well as it is not in some bound period is a free period. We  will define precisely all these notions below.

\subsection{The initial interval} \label{sec2.1}
Here we work to acquire the starting interval of parameters where we initiate the inductive construction. 
The next lemma  provides useful properties for maps near $f_{0}$; see e.g. \cite[Lemma 2.1]{AlSo2012} for a proof.
\begin{lemma} \label{lem4.1.1}  
There is $\lambda_{c}>1$ and a large integer $\Delta_{c}$ such that for any $\Delta \ge \Delta_{c}$ there are $a_{0}^{\prime}>0$ and $c>0$ such that for all  $x\in I$ and $a\in [0,a_{0}^{\prime}]$ we have:
\begin{enumerate}
\item [(1)] if $x, f_{a}(x), ..., f_{a}^{n-1}(x)\notin(-e^{-\Delta},e^{-\Delta})$, then $(f_{a}^{n})^{\prime}(x) \ge c\lambda_{c}^{n};$
\item [(2)] if $x, f_{a}(x), ..., f_{a}^{n-1}(x)\notin(-e^{-\Delta},e^{-\Delta})$ and $f_{a}^{n}(x)\in (-e^{-\Delta},e^{-\Delta})$, then $(f_{a}^{n})^{\prime}(x)\ge \lambda_{c}^{n}$;
\item [(3)] if $x, f_{a}(x), ..., f_{a}^{n-1}(x) \notin (-e^{-\Delta},e^{-\Delta})$ and $f_{a}^{n}(x)\in(-e^{-1}, e^{-1})$, then $(f_{a}^{n})^{\prime}(x)\ge \frac{1}{e} \lambda_{c}^{n}$.
\end{enumerate}
\end{lemma}
The following result is based on the fact that the maps $\xi_{k}^{\pm}$ are differentiable as long as they stay away from 0, and states that under strong growth of the derivatives of $f_{a}$ at the critical values $\pm1$ the parameter and the space derivatives are comparable. 
\begin{proposition}\label{prop4.1.2}
Given $\lambda>1$ and $\eta>2$, there are $N^{\pm}\ge 2$ and $A^{\pm}>0$ such that if $a \ge 0$ and $n \ge N^{\pm} $ satisfy both
\begin{enumerate}
\item  $D_{j}^{\pm}(a) \ge \eta^{j} $ for  $1\le j\le N^{\pm}$, and
\item  $D_{j}^{\pm}(a) \ge \lambda^{j}$,  for $1\le j\le n-1$,
\end{enumerate}
then
$$
\frac{1}{A^{\pm}} \le \frac{|(\xi_{n}^{\pm})^{\prime}(a)|}{D_{n-1}^{\pm}(a)} \le A^{\pm}. 
$$
\end{proposition}
\begin{proof}
We consider the case of the critical value $-1$, the case of $+1$ is similar. Setting $f(a,x)=f_{a}(x)$ and using the chain rule for $k\ge 1$, we have
\begin{equation} \label{4.1.1}
D_{k}^{+}(a)= \frac{\partial f}{\partial x}(a,\xi_{k}^{+}(a))\cdot D_{k-1}^{+}(a)
= \prod_{i=1}^{k} \frac{\partial f}{\partial x}(a,\xi_{i}^{+}(a)).
\end{equation}
On the other hand,
\begin{align} \label{4.1.2}
(\xi_{k+1}^+)^{\prime}(a) &= \frac{\partial f}{\partial x}(a,\xi_{k}^{+}(a))\cdot (\xi_{k}^{+})^{\prime}(a)+ \frac{\partial f}{\partial a}(a,\xi_{k}^{+}(a)) \nonumber\\
&= \frac{\partial f}{\partial x}(a,\xi_{k}^{+}(a))  [\frac{\partial f}{\partial x}(a,\xi_{k-1}^{+}(a))\cdot (\xi_{k-1}^{+})^{\prime}(a)\nonumber \\
& \mbox{ } \mbox{ } + \frac{\partial f}{\partial a}(a,\xi_{k-1}^{+}(a))] + \frac{\partial f}{\partial a}(a,\xi_{k}^{+}(a)) \nonumber\\
&= \frac{\partial f}{\partial x}(a,\xi_{k}^{+}(a)) \frac{\partial f}{\partial x}(a,\xi_{k-1}^{+}(a) [\frac{\partial f}{\partial x}(a,\xi_{k-2}^{+}(a)\cdot (\xi_{k-2}^{+})^{\prime}(a)\nonumber\\
& \mbox{ } \mbox{ }  + \frac{\partial f}{\partial a}(a,\xi_{k-2}^{+}(a))] + \frac{\partial f}{\partial x}(a,\xi_{k}^{+}(a))  \frac{\partial f}{\partial a}(a,\xi_{k-1}^{+}(a)) +\mbox{ }\frac{\partial f}{\partial a}(a,\xi_{k}^{+}(a)) \nonumber\\
&= \prod_{i=1}^{k} \frac{\partial f}{\partial x}(a,\xi_{i}^{+}(a)) \cdot (\xi_{1}^{+})^{\prime}(a) + \prod_{i=2}^{k} \frac{\partial f}{\partial x}(a,\xi_{i}^{+}(a)) \frac{\partial f}{\partial a}(a,\xi_1^{+}(a)) \nonumber
\\
& \mbox{ } \mbox{ } + \ldots + \frac{\partial f}{\partial x}(a,\xi_{k}^{+}(a)) \frac{\partial f}{\partial a}(a,\xi_{k-1}^{+}(a)) + \frac{\partial f}{\partial a}(a,\xi_{k}^{+}(a)).
\end{align}
From \eqref{4.1.1} and \eqref{4.1.2} we get
\begin{align}\label{4.1.3}
\frac{(\xi_{k+1}^{+})^{\prime}(a)}{D_{k}^{+}(a)} - \frac{(\xi_{k}^{+})^{\prime}(a)}{D_{k-1}^{+}(a)} = \frac{\dfrac{\partial f}{\partial a}(a,\xi_{k}^{+}(a))}{\prod_{i=1}^k \frac{\partial f}{\partial x}(a,\xi_{i}^{+}(a))} = \frac{\dfrac{\partial f}{\partial a}(a,\xi_{k}^{+}(a))}{D_{k}^{+}(a)}.
\end{align}
Summing both sides of \eqref{4.1.3} over $k = 1,...,n-1$, we obtain
$$
\frac{(\xi_{n}^{+})^{\prime}(a)}{D_{n-1}^{+}(a)} - \frac{(\xi_{1}^{+})^{\prime}(a)}{D_{0}^{+}(a)} =  \sum_{k=1}^{n-1}\frac{\dfrac{\partial f}{\partial a}(a,\xi_{k}^{+}(a))}{D_{k}^{+}(a)}.
$$
We may assume that there exist $A_{1}, A_{2}>0$ such that for every parameter $a$, 
$$
A_{1}<\underset{x\in I}{\sup} \mbox{ } \left|\dfrac{\partial f}{\partial a}(a,x)\right| \le |(\xi_{1}^{+})^{\prime}(a)| \le A_{2}.
$$ 
Since $D_{0}^{+}(a)=1$,   we get
\begin{align*}
\left|\left|\frac{(\xi_{n}^{+})^{\prime}(a)}{D_{n-1}^{+}(a)}\right|-\left|(\xi_{1}^{+})^{\prime}(a)\right|\right| & \le  \left|\frac{(\xi_{n}^{+})^{\prime}(a)}{D_{n-1}^{+}(a)}-(\xi_{1}^{+})^{\prime}(a)\right|\\
&=\left|\sum_{k=1}^{n-1}\frac{\dfrac{\partial f}{\partial a}(a,\xi_{k}^{+}(a))}{D_{k}^{+}(a)}\right|\\
& \le  \underset{x\in I}{\sup} \mbox{ } \left|\dfrac{\partial f}{\partial a}(a,x)\right|\sum_{k=1}^{n-1}\frac{1}{D_{k}^{+}(a)}\\
& \le  \left|(\xi_{1}^{+})^{\prime}(a)\right|\sum_{k=1}^{n-1}\frac{1}{D_{k}^{+}(a)}.\\
\end{align*}
It follows that
\begin{align}\label{4.1.4}
A_{1} \left(1-\sum_{k=1}^{n-1}\frac{1}{D_{k}^{+}(a)}\right) \le \frac{|(\xi_{n}^{+})^{\prime}(a)|}{D_{n-1}^{+}(a)} \le  A_{2} \left(1+ \sum_{k=1}^{n-1}\frac{1}{D_{k}^{+}(a)}\right).
\end{align}
On the other hand, since $\eta>2$ and $\lambda>1$,  we can choose an integer $N_{0}^{+}$ and   $\epsilon^{\prime}>0$ such that  $$\sum_{k=1}^{+\infty}\frac{1}{\eta^{k}} + \sum_{k=N_{0}^{+}+1}^{+\infty} \frac{1}{\lambda^{k}}<1- \epsilon^{\prime}.$$
Thus,  if $D_{k}^{+}(a)\ge \eta^{k}$ for every $k=1,\ldots,N_{0}^{+}$, and $D_{k}^{+}(a)\ge \lambda^{k}$ for every $k=N_{0}^{+}+1,\ldots, n-1$, we obtain
\begin{align*}
\sum_{k=1}^{n-1}\frac{1}{D_{k-1}^{+}(a)}&\le  \sum_{k=1}^{N_{0}^{+}}\frac{1}{\eta^{k}}+\sum_{k=N_{0}^{+}+1}^{n-1}\frac{1}{\lambda^{k}} \\
&\le  \sum_{k=1}^{\infty}\frac{1}{\eta^{k}}+\sum_{k=N_{0}^{+}+1}^{\infty}\frac{1}{\lambda^{k}} \\
&\le 1- \epsilon^{\prime}.
\end{align*}
The result follows  from  \eqref{4.1.4} with $A^{+} \ge \max \left\{ \dfrac{1}{\epsilon^{\prime}  A_{1}},  A_{2}(2- \epsilon^{\prime}) \right\}$.
\end{proof}

From here on we take
\begin{equation}\label{eq.NA}
N=\max \{ N^{+}, N^{-} \}\quad\text{and}\quad A=\max \{ A^{+}, A^{-} \},
\end{equation}
where $N^{\pm}$ and $A^{\pm}$ are provided by Proposition~\ref{prop4.1.2}. 
\begin{remark} \label{rmk4.1.3}
Observe that if conditions (1) and (2) of Proposition~\ref{prop4.1.2} are satisfied for some $n\ge N$ and for every $a$ in some parameter interval $\omega$, then we have in particular $\xi_{k}^{\pm}(a)\neq 0$ for all $a\in \omega$ and $N\le k\le n$. Then for any $N \le k \le n$, the maps $\xi_{k}^{\pm}|_{\omega}$ are diffeomorphisms with the inverses defined as: for any $x^{\pm} \in \xi_{k}^{\pm}(\omega)$ with $\xi_{k}^{\pm}(a)=x^{\pm}$ for some $a\in\omega$, then
$$
(\xi_{k}^{\pm})^{-1}(x^{\pm}):=\xi_{-k}^{\pm}(x^{\pm})=a.
$$
In fact, $\xi_{k}^{\pm}|_{\overline{\omega}}$ are diffeomorphisms and this assertion plays an important part to inductively construct the set of Rovella parameters. Consequently, for every $N \le i \le j \le n$, we can define the following functions
$$\begin{array}{c}
\psi^{\pm}: \mbox{ } \xi_{i}^{\pm}(\omega) \longrightarrow \mbox{ }\mbox{ }\mbox{ }\mbox{ } \xi_{j}^{\pm}(\omega) \\[2pt]
 \mbox{ }\mbox{ }\mbox{ }\mbox{ }\mbox{ }\mbox{ }\mbox{ }\mbox{ }\mbox{ }\mbox{ }\mbox{ }\mbox{ }\mbox{ }\mbox{ }\mbox{ }\mbox{ }\mbox{ }\mbox{ }\mbox{ } x \mbox{ }\mbox{ } \mbox{ }\mapsto \mbox{ }\mbox{ } \xi_{j}^{\pm}\circ (\xi_{i}^{\pm})^{-1}(x),
\end{array}$$
with the derivative given for $a\in\omega$ by
$$
(\psi^{\pm})^{\prime}(\xi_{i}^{\pm}(a))=\frac{(\xi_{j}^{\pm})^{\prime}(a)}{(\xi_{i}^{\pm})^{\prime}(a)}.
$$
\end{remark}
The functions $\psi^{\pm}$ will be useful in the proof of the next lemma which will be used later in finding an estimate for the lengths of $\xi_{n}^{\pm}(\omega)$, where $\omega$ is a parameter interval. For an interval $J \in \mathbb{R}$, we dente by $|J|$ as usual length of $J$.
\begin{lemma}\label{lem4.1.4}
Given $\lambda>1$ and $\eta>2$, consider a parameter interval $\omega$ such that for every $a\in\omega$ and some $n\ge N$ hold both
 \begin{enumerate}
\item  $D_{j}^{\pm}(a) \ge \eta^{j} $ for  $1\le j\le N $, and
\item  $D_{j}^{\pm}(a) \ge \lambda^{j}$,  for $1\le j\le n-1$.
\end{enumerate}
Then, for  any $N \le i \le j \le n$, there is $a^\pm\in\omega$ such that 
$$\dfrac{1}{A^{2}}\,  \left|(f_{a^\pm}^{j-i})^{\prime}(\xi_{i}^{\pm}(a^\pm))\right| \le \dfrac{|\xi_{j}^{\pm}(\omega)|}{|\xi_{i}^{\pm}(\omega)|} \le A^{2} \left|(f_{a^\pm}^{j-i})^{\prime}(\xi_{i}^{\pm}(a^\pm))\right|.$$
\end{lemma}
\begin{proof}
We are going to present the proof corresponding to critical value $-1$, the other case being similar. 
Since properties (1) and (2) hold for every $a\in\omega$, it follows from Proposition~\ref{prop4.1.2} that
\begin{align}\label{4.1.5}
\frac{1}{A^{2}}\cdot \frac{D_{j-1}^{+}(a)}{D_{i-1}^{+}(a)}  \le \frac{|(\xi_{j}^{+})^{\prime}(a)|}{|(\xi_{i}^{+})^{\prime}(a)|} \le A^{2}\cdot \frac{D_{j-1}^{+}(a)}{D_{i-1}^{+}(a)}.
\end{align}
On the other hand, by the Mean Value Theorem, for some ${a^+}\in\omega$ we have
\begin{align}\label{4.1.6}
\frac{|\xi_{j}^{+}(\omega)|}{|\xi_{i}^{+}(\omega)|} = |(\xi_{j-i}^{+})^{\prime}(\xi_{i}^{+}({a^+}))| = |(\xi_{j}^{+} \circ \xi_{-i}^{+})^{\prime}(\xi_{i}^{+}({a^+}))| = |(\psi^{+})^{\prime}(\xi_{i}^{+}({a^+}))|.
\end{align}
Also
\begin{align*}
D_{j-1}^{+}({a^+})= (f_{{a^+}}^{j-1})^{\prime}(-1) &=  (f_{{a^+}}^{j-i} \circ f_{{a^+}}^{i-1})^{\prime}(-1) \nonumber \\
& = (f_{{a^+}}^{j-i})^{\prime}(f_{{a^+}}^{i-1}(-1)) (f_{{a^+}}^{i-1})^{\prime}(-1) \nonumber \\
&   =  (f_{{a^+}}^{j-i})^{\prime}(\xi_{i}^{+}({a^+})) D_{i-1}^{+}({a^+}),
\end{align*}
which gives
\begin{align}\label{4.1.7}
\frac{D_{j-1}^{+}({a^+})}{D_{i-1}^{+}({a^+})} = (f_{{a^+}}^{j-i})^{\prime}(\xi_{i}^{+}({a^+})).
\end{align}
Now using \eqref{4.1.6} and \eqref{4.1.7} in \eqref{4.1.5}, we get
$$\dfrac{1}{A^{2}}\,  \left|(f_{a^+}^{j-i})^{\prime}(\xi_{i}^{+}(a^+))\right| \le \dfrac{|\xi_{j}^{+}(\omega)|}{|\xi_{i}^{+}(\omega)|} \le A^{2} \left|(f_{a^+}^{j-i})^{\prime}(\xi_{i}^{+}(a^+))\right|,$$
and so the result follows.
\end{proof}
The next proposition provides the initial interval of our  construction of the parameter sets. Recall that $N$ is given in~\eqref{eq.NA} and the constants $\lambda_{c}>1, \Delta_c\in\mathbb N$ and     $a_0'>0$ are given in Lemma~\ref{lem4.1.1}.
\begin{proposition}\label{prop4.1.5}
There exist $1< \lambda_{0} \le \lambda_{c}$, $\eta_{1}>2$ and $\Delta \ge \Delta_{c}$ such that given any integer $N_{0}\ge N$, there exist an integer $N_{1} \ge N_{0}$ and a parameter $0<a_{0} \le a_{0}^{\prime}$ for which
\begin{enumerate}
\item $D_{j}^{+}(a)\ge \eta_{1}^{j} \mbox{ } for \mbox{ } every \mbox{ } a\in [0,a_{0}] \mbox{ } and \mbox{ } 1\le j\le N_{0}-1$,
\item $D_{j}^{+}(a)\ge \lambda_{0}^{j} \mbox{ } for \mbox{ } every \mbox{ } a\in [0,a_{0}] \mbox{ } and \mbox{ } 1\le j\le N_{1}-1$,
\item $\xi_{j}^{+}\left([0,a_{0}]\right)\cap (-e^{-\Delta},e^{-\Delta})=\phi \mbox{ } for \mbox{ } every \mbox{ } 1\le j\le N_{1}-1$,
\item $\xi_{N_{1}}^{+}\left([0,a_{0}]\right)\supset (-e^{-\Delta},e^{-\Delta})$.
\end{enumerate}
\end{proposition}
\begin{proof}
For each $1 \le n \le N_{0}$, consider the map
$
\Phi_{n}: \mbox{ } [0,a_{0}^{\prime}] \longrightarrow [-1,1]\times [0,+ \infty) $
given by $$ \Phi_n(a)= (\xi_{n+1}^{+}(a),D_{n}^{+}(a)).
$$
Since the point $-1$ is fixed by $f_{0}$, using the chain rule we get
\begin{equation}\label{eq.dn}
D_{n}^{+}(0)=(f_{0}^{n})^{\prime}(-1)= \prod_{i=0}^{n-1}f_{0}^{\prime}(f_{0}^{i}(-1))=\prod_{i=0}^{n-1}f_{0}^{\prime}(-1).
\end{equation}
From the properties of the map $f_{0}$, we may choose $\eta_{0}>2$ and $\epsilon_{0}>0$ such that $f_{0}^{\prime}(-1) = \eta_{0}$ and $\eta_{0}-\epsilon_{0}>2$. We set $\eta_{1}=\eta_{0}-\epsilon_{0}$ and denote $O^{-}(a)\in [-1,0)$ the zero of the map $f_{a}$. From \eqref{eq.dn} we have $\Phi_{n}(0) = (-1,\eta_{0}^{n})$. Since $\Phi_{k}$ is continuous as long as $\xi_{k}^{+}$ is not mapped onto the origin,  we have  parameters   $a_1\ge a_2\ge\cdots \ge a_{N_0}$  such that for all $1 \le n \le N_{0}$
$$
\Phi_{n}\left([0,a_{n}]\right)\subset [-1,O^{-}(0)]\times [\eta_{1}^{n},+\infty).
$$
That is, for every $1 \le n \le N_{0}$ and every $a\in [0,a_{N_{0}}]$ we have
$$ 
\xi_{n+1}^{+}(a) \le O^{-}(0)\quad\text{and}\quad D_{n}^{+}(a) \ge \eta_{1}^{n}.
$$
Thus, any $a\in [0,a_{N_{0}}]$ satisfies the first item. 
On the other hand, since 1 is a critical value for~$f_{0}$ with $f_{0}(0^{-})=1$, $O^{-}(0) < -1$ and $f_{0}^{\prime}(x)\le f_{0}^{\prime}(y)$ for $x,y\in [-1,0)$ with $x\ge y$, we may find  $\lambda_{0}^{\prime}>1$ and a large number $\Delta_{0}$ such that $f_{0}^{\prime}(x_{0}) \ge \lambda_{0}^{\prime}$ and $f_{0}(x_{0})>e^{-\Delta_{0}}$, for some $x_{0}\in (O^{-}(0),0)$. Setting $\lambda_{0} = \min \{ \lambda_{c},  \lambda_{0}^{\prime} \}$ and $\Delta = \max \{ \Delta_{c}, \Delta_{0} \}$, then for any parameter~$a$, if $\xi_{j}^{+}(a) \in [-1,x_{0}]$ for every $j=1, \ldots, k$, we have
$
D_{k}^{+}(a)\ge \lambda_{0}^{k}.
$
Now, as long as $\xi_{i}^{+}\left([0,a_{N_{0}}]\right)$ belongs to  $[-1,x_{0})$, any $a\in [0,a_{N_{0}}]$ satisfies the hypothesis of Proposition~\ref{prop4.1.2}. Thus, using the Mean Value Theorem, for some $a\in (0,a_{N_{0}})$, we have
\begin{equation*}
|\xi_{i+1}^{+}\left([0,a_{N_{0}}]\right) | = |(\xi_{i+1}^{+})^\prime (a)| a_{N_{0}}
 \ge   \frac{a_{N_{0}}}{A} D_{i}^{+}(a) 
 \ge  \frac{a_{N_{0}}}{A} \lambda_{0}^{i}.
\end{equation*}
The above inequality reveals that while $\xi_{i}^{+}\left([0,a_{N_{0}}]\right)$ remains inside the interval $[-1,x_{0})$, we have exponential growth for $\xi_{i}^{+}\left([0,a_{N_{0}}]\right)$, and then there exists an integer $k$ such that $\xi_{k}^{+}\left([0,a_{N_{0}}]\right) \not \subset [-1,x_{0})$. Let $N_{1}^{\prime}$ be the first integer in that situation, i.e.
$$
\xi_{i}^{+}\left([0,a_{N_{0}}]\right) \subset [-1,x_{0}), \quad{}  \mbox{ for every } 1\le i <N_{1}^{\prime},
$$
and
$$
\xi_{N_{1}^{\prime}}^{+}\left([0,a_{N_{0}}]\right) \not \subset [-1,x_{0}).
$$
Therefore, we may chose $a_{0} \in [0,a_{N_{0}}]$ such that $\xi_{N_{1}^{\prime}}^{+}(a_{0})=x_{0}$, and since $f_{a_{0}}(x_{0}) \ge e^{-\Delta}$, then $\xi_{N_{1}^{\prime}+1}^{+}([0,a_{0}]) \supset [-1,e^{-\Delta})$. Taking $N_{1}=N_{1}^{\prime}+1$ the result follows.
\end{proof}
\begin{remark}\label{rmk4.1.7}
From property (A0), we know that the points $1$ and $-1$ are fixed by the map~$f_{0}$, therefore by the definition of $f_{0}$, it can be seen that the connected components of the graph of $f_{0}$ in the intervals $[-1,0)$ and $(0,1]$ are symmetric about origin, i.e. $f_{0}(x)=-f_{0}(-x)$ for all $x\in I\setminus \{ 0 \}$. For the sake of simplicity we may assume that for any parameter $a$ corresponding to contracting Lorenz family, $f_{a}(x)=-f_{a}(-x)$ for all $x\in I\setminus \{ 0 \}$. Thus, a result  similar to Proposition~\ref{prop4.1.5} can be obtained for $\xi^{-}$ and $D^{-}$ with the same integer~$N_{1}$ and the parameter interval $[0,a_{0}]$. However, we also remark that the results can be proved in more general setting without the assumption of symmetry.
\end{remark}

\subsection{Bound periods} \label{sec2.2}
The periods of time occurring after the returns of critical orbits $\xi_{k}^{\pm}(a)$ to a small neighborhood of $0$ have a significant role. 
In order to explicitly describe the closeness to $0$, we set $\delta=e^{-\Delta}$, where $\Delta$ is given in Proposition \ref{prop4.1.5}. 
 We start by fixing some $\alpha>0$ such that
 $$c':=1-\left(2\alpha +\frac{1}{\ln \lambda_{0}}(s-1)\alpha\right) >0,$$
with $\lambda_0>1$  given by Proposition~\ref{prop4.1.5}, and   define
\begin{equation}\label{eq.lambda}
\lambda = \lambda_{0}^{c^{\prime}} > 1.
\end{equation}
We may take $\alpha$ sufficiently small such that $\alpha s < \ln\lambda$. 
If necessary, we make $\alpha$ smaller and fix some $\beta>0$ such that 
\begin{equation}\label{eq.beta}
s \alpha \le \beta\quad\text{and}\quad\beta \dfrac{s+5}{\beta+\log\lambda} < 1.
\end{equation}
Observe that we have $\lambda\to\lambda_0$ when $\alpha\to0$, which makes possible   all these choices. 

Next we consider for $m \ge \Delta - 1$ the neighborhoods of $0$ 
$$
U_{m}=(-e^{-m},e^{-m}) $$
and the sets 
$$
I_{m}=[e^{-(m+1)},e^{-m}) \quad\mbox{and} \quad I_{m}^{+}=I_{m-1} \cup I_{m} \cup I_{m+1}.
$$
We also consider the above sets  for $m \le -(\Delta - 1)$, defining
 $$I_{m}=-I_{|m|}\quad\text{and}\quad I_{m}^{+}=-I_{|m|}^{+}.$$  

\begin{definition}  \label{bp(a)}
Given $x \in I_{m}^{+}$, denote by $p(a,m)$ to be the largest integer such that 
$$
|f_{a}^{j}(x)- \xi_{j}^{+}(a)| \le e^{-\beta j}, \quad  \mbox{if } m>0,
$$
and
$$
|f_{a}^{j}(x)- \xi_{j}^{-}(a)| \le e^{-\beta j}, \quad \mbox{if  } m<0,
$$
for $j=1, \ldots, p(a,m)$. The time interval $1, \ldots, p(a,m)$ is called the \emph{bound period} for $x$.
\end{definition}
Note that by this definition we have for all $1\le j \le p(a,m)$
$$
|f_{a}^{j-1}\big([-1,f_{a}(e^{-|m|+1})]\big)| \le e^{-\beta j},
$$
In our next result we state the key properties of these periods. 
Recall that $R_{n} \subset [0,a_{0}]$ is a set satisfying \eqref{(BAn)} and \eqref{(EGn)}, and according to Remark~\ref{rmk4.1.7}, if $a\in R_{n-1}$ and $\xi_{n}^{+}(a)\in I_{m}^{+}$ for some $m$ with $|m| \ge \Delta$, then $\xi_{n}^{-}(a)\in I_{-m}^{+}$ and $p(a,m)=p(a,-m)$.

\begin{lemma}\label{lem4.2.2}
Assume that $a\in R_{n-1}$ and either $\xi_{n}^{+}(a)$ or $\xi_{n}^{-}(a)$ belongs to an interval $I_{m}^{+}$, for some $\Delta \le |m| \le [\alpha n]-1$. Then   
\begin{enumerate}
\item there exists  $B_{1}=B_{1}(\alpha , \beta)$ such that for every $k=1, \ldots, p(a,m)$
\begin{enumerate}
\item 
$\displaystyle
\frac{1}{B_{1}} \le \frac{(f_{a}^{k})^{\prime}(y)}{D_{k}^{+}(a)} \le B_{1} $ if $ y \in \left[-1,f_{a}(e^{-|m|+1})\right],
$
\item
$\displaystyle
\frac{1}{B_{1}} \le \frac{(f_{a}^{k})^{\prime}(y)}{D_{k}^{-}(a)} \le B_{1}$ if $ y \in \left[f_{a}(-e^{-|m|+1}),1\right];
$
\end{enumerate}
\item 
$\displaystyle
p(a,m) \le \frac{s+1}{\beta + \log \lambda}|m|;
$
\item  
letting  $\kappa_{1}= \beta \dfrac{s+2}{\beta+ \log \lambda}$, we have for all $x \in I_{m}^{+}$ and $p=p(a,m)$ 
$$
(f_{a}^{p+1})^{\prime}(x) \ge e^{ (1 - \kappa_{1}) |m|}.
$$
\end{enumerate}
\end{lemma}
\begin{proof}
For obtaining (1) it is sufficient to prove the first item, for the second one can be obtained following similar lines. We may assume that $\xi_{n}^{+}(a) \in I_{m}^{+}$. First using chain rule, for $k=1, \dots,  \min{\{p,n}\}$, we have
\begin{align*}
\frac{(f_{a}^{k})^{\prime}(y)}{D_{k}^{+}(a)} &= \frac{(f_{a}^{k})^{\prime}(y)}{(f_{a}^{k})^{\prime}(-1)} = \prod_{j=0}^{k-1} \frac{f_{a}^{\prime}(f_{a}^{j}(y))}{f_{a}^{\prime}(\xi_{j+1}^{+}(a))}\\
& = \prod_{j=0}^{k-1} \left(1+ \frac{f_{a}^{\prime}(f_{a}^{j}(y)) - f_{a}^{\prime}(\xi_{j+1}^{+}(a))}{f_{a}^{\prime}(\xi_{j+1}^{+}(a))}\right)\\
& \le \exp\left( \sum_{j=0}^{k-1} \left|\frac{f_{a}^{\prime}(f_{a}^{j}(y)) - f_{a}^{\prime}(\xi_{j+1}^{+}(a))}{f_{a}^{\prime}(\xi_{j+1}^{+}(a))}\right| \right).
\end{align*}
Therefore we conclude the proof of this item by showing that
$$
\sum_{j=0}^{k-1} \frac{|f_{a}^{\prime}(f_{a}^{j}(y)) - f_{a}^{\prime}(\xi_{j+1}^{+}(a))|}{f_{a}^{\prime}(\xi_{j+1}^{+}(a))}
$$
is uniformly bounded. Since $0$ is not in $[\xi_{j}^{+}(a)-e^{-\beta j}, \xi_{j}^{+}(a)+e^{-\beta j}]$ and $f_{a}$ has negative Schwarzian derivative inside this interval, as long as $f_{a}^{j}(y) \in [\xi_{j}^{+}(a)-e^{-\beta j}, \xi_{j}^{+}(a)+e^{-\beta j}],$
\begin{align*}
\frac{|f_{a}^{\prime}(f_{a}^{j}(y)) - f_{a}^{\prime}(\xi_{j+1}^{+}(a))|}{f_{a}^{\prime}(\xi_{j+1}^{+}(a))} & \le |f_{a}^{\prime \prime}(z)| \frac{|f_{a}^{j}(y)-\xi_{j+1}^{+}(a)|}{f_{a}^{\prime}(\xi_{j+1}^{+}(a))}\\
& \le C|z|^{s-2} \frac{|f_{a}^{j}(y)-\xi_{j+1}^{+}(a)|}{f_{a}^{\prime}(\xi_{j+1}^{+}(a))}.
\end{align*}
Now $k \le n,p$ and $a$ satisfies (BA$_{n-1})$, therefore from the above inequality, using the binding condition and property (A3), we get
$$
\sum_{j=0}^{k-1} \frac{|f_{a}^{\prime}(f_{a}^{j}(y)) - f_{a}^{\prime}(\xi_{j+1}^{+}(a))|}{f_{a}^{\prime}(\xi_{j+1}^{+}(a))}  \le \frac{C}{K_0} \sum_{j=0}^{k-1} \frac{e^{-\beta j}}{e^{-\alpha(s-1)(j+1)}}.
$$
The right side of the above inequality is uniformly bounded since $\beta \ge s \alpha$ with $s>1$. Consequently to conclude the proof of (1) we just need to make sure that $p<n$. See part~(2).

For proving (2), let $x=e^{-|m|+1} \in I_{m}^{+}$ and $j=\min{\{p,n\}}-1$. Then using the first part of~(1)  and property (A3), we have
\begin{align*}
|f_{a}^{j+1}(x)-\xi_{j+1}^{+}(a)| &= |f_{a}^{j}(f_{a}(x))-f_{a}^{j}(-1)|\\
&= (f_{a}^{j})^{\prime}(y)|f_{a}(x)+1|, \mbox{ } y \in (-1,f_{a}(e^{-|m|+1}))\\
& \ge \frac{K_0}{B_{1}}D_{j}^{+}(a)\frac{|x|^{s}}{s}.
\end{align*}
Now, using the binding condition and taking into account that $a$ satisfies (EG$_{n-1})$, from the last inequality it follows that
$$
\frac{K_0 }{B_{1}s} \lambda^{j} e^{-(|m|+2)s}  \le e^{-\beta (j+1)},
$$
and from the above inequality it can be work out that
$$
j  \le \frac{|m|s}{\beta + \log \lambda} + \frac{2s-\log(\frac{K_0}{B_{1}s})-\beta}{\beta + \log \lambda}.
$$
Therefore if $|m|$ is large enough, we may conclude that
\begin{align}\label{4.2.1}
j & \le \frac{|m|(s+1)}{\beta + \log \lambda}  - 1.
\end{align}
Since $|m| \le [\alpha n]-1$,   from \eqref{4.2.1} we have
\begin{align*}
j & \le \frac{([\alpha n]-1)(s+1)}{\beta + \log \lambda}  - 1 \le \frac{(\alpha n-1)(s+1)}{\beta + \log \lambda}  - 1\\
& \le \frac{(\alpha n)(s+1)}{\beta + \log \lambda}  - 1 < n-1,
\end{align*}
where the last inequality holds since $\beta \ge s \alpha$ and $\alpha<\log \lambda$. Hence $j=p-1$ and from \eqref{4.2.1} the result follows.

Let us now prove (3).  Clearly, by the binding condition
\begin{align}\label{4.2.2}
|f_{a}^{p}\big([-1,f_{a}(e^{-|m|+1})]\big)| & \ge e^{-\beta (p+1)}.
\end{align}
Thus by the Mean Value Theorem, for some $z \in (-1,f_{a}(e^{-|m|+1}))$ and for some $y \in (0,e^{-|m|+1})$, we have
\begin{align}\label{4.2.3}
|f_{a}^{p}\big([-1,f_{a}(e^{-|m|+1})]\big)|  &= (f_{a}^{p})^{\prime}(z)f_{a}^{\prime}(y)e^{-|m|+1}.
\end{align}
From \eqref{4.2.2} and \eqref{4.2.3}, we obtain
$$
(f_{a}^{p})^{\prime}(z)  \ge \frac{e^{-\beta (p+1)+|m|-1}}{f_{a}^{\prime}(y)}.
$$
Using the above inequality, property (A3) and part (1), for any $x \in I_{m}^{+}$, we get
\begin{align*}
(f_{a}^{p+1})^{\prime}(x) & = (f_{a}^{p})^{\prime}(f_{a}(x))f_{a}^{\prime}(x)\\
& \ge \frac{1}{B_{1}} D_{p}^{+}(a) f_{a}^{\prime}(x), \mbox{ since $f_{a}(x) \in [-1,f_{a}(e^{-|m|+1})]$ }\\
& \ge \frac{1}{B_{1}^{2}}(f_{a}^{p})^{\prime}(z) f_{a}^{\prime}(x), \mbox{ since $z \in [-1,f_{a}(e^{-|m|+1})]$ }\\
& \ge \frac{1}{B_{1}^{2}}e^{-\beta (p+1)+|m|-1}\cdot \frac{f_{a}^{\prime}(x)}{f_{a}^{\prime}(y)}\\
& \ge \frac{1}{B_{1}^{2}}e^{-\beta (p+1)+|m|-1}\cdot \frac{K_0|x|^{s-1}}{K_{1}|y|^{s-1}}.
\end{align*}
Since $|x|\ge e^{-|m|-2},$ $|y| \le e^{-|m|+1}$ and from part $(2)$ we have $p<\frac{s+1}{\beta+\log \lambda}|m|$. Hence the result follows from the above inequality, provided $\Delta$ is sufficiently large so that 
$$\frac{K_0}{K_{1}B_{1}^{2}}e^{-(3s+\beta-2)} \ge e^{-\frac{\beta}{\beta+\log \lambda}|m|}.$$
\end{proof}

Now we are intended to find similar bounds, as in the above lemma, when $p(a,m)$ is constant in small parameter intervals. We start with some preliminary results that culminate the main goal of this subsection, Proposition~\ref{lem4.2.6}. In this regard, for a parameter interval $\omega$ such that either $\xi_{n}^{+}(\omega)$ or $\xi_{n}^{-}(\omega)$ is contained in some $I_{m}^{+}$, with $|m| \ge \Delta$ we define
$$
p(\omega,m)= \underset{a \in \omega}{\min} \mbox{ }  p(a,m).
$$
Note that by the above definition $p(\omega,m) \le p(a,m)$ and
$$
|f_{a}^{j-1}\big([-1,f_{a}(e^{-|m|+1}))\big)| \le e^{-\beta j},
$$
for all $1\le j \le p(\omega,m)$ and for every $a\in\omega$. Furthermore, $p(\omega,m)=p(\omega,-m)$ and $p(\omega,m) \le p(a,m)$, therefore for every $a \in \omega$ items (1) and (2) of Lemma \ref{lem4.2.2} follow directly. But it requires some more work in order to prove part (3) and this is what we are going to establish in the remaining section. 

\begin{lemma}\label{lem4.2.4}
If $\omega \subset R_{n-1}$ is an interval such that either $\xi_{n}^{+}(\omega)$ or $\xi_{n}^{-}(\omega)$ is contained in $I_{m}^{+}$  with $\Delta \le |m| \le [\alpha n]-1$, then for every $a,b \in \omega$ and every $1 \le j \le p(\omega,m)$ we have
$$
\big| |\xi_{j}^{\pm}(a)|^{s-1}-|\xi_{j}^{\pm}(b)|^{s-1} \big| \le e^{-\beta j}.
$$
\end{lemma}
\begin{proof}
We prove the result  in the case of $\xi_{j}^{+}$, the other one can be proved similarly. If $a=b$ then it is trivial. So let us assume $a \neq b$. From inequality \eqref{4.1.4} in the proof of Proposition~\ref{prop4.1.2}, we have
$$
\frac{|(\xi_{j+1}^{+})^{\prime}(a)|}{D_{j}^{+}(a)} \le A_{2} (1+ \sum_{k=1}^{j}\frac{1}{D_{k-1}^{+}(a)}),
$$
and since $\omega \subset R_{n-1}$ and $j\le p(\omega,m) \le n-1$, we get
$$
\frac{|(\xi_{j+1}^{+})^{\prime}(a)|}{D_{j}^{+}(a)}  \le A_{2} (1+ \sum_{k=1}^{j}\frac{1}{\lambda^{k-1}})
\le A_{2} (1+ \sum_{k=1}^{\infty}\frac{1}{\lambda^{k-1}}) \le A_{3},
$$
for some $A_{3}>0$.
Now, if $1<s\le2$, since the modulus function is differentiable everywhere but $0$, using the above inequality and the Mean Value Theorem, we get
\begin{align}\label{4.2.4(i)}
\big| |\xi_{j}^{+}(a)|^{s-1}-|\xi_{j}^{+}(b)|^{s-1} \big| & \le \big| |\xi_{j}^{+}(a)|-|\xi_{j}^{+}(b)| \big| \nonumber \\
& = \left| \frac{\xi_{j}^{+}(d)}{|\xi_{j}^{+}(d)|} (\xi_{j}^{+})^{\prime}(d) \right| |a-b|, \quad d \in (a,b)\nonumber \\
& \le \frac{ |(\xi_{j}^{+})^{\prime}(d)|}{D_{j-1}^{+}(d)} D_{j-1}^{+}(d) |a-b| \nonumber \\
& \le  A_{3} D_{j-1}^{+}(d) |a-b|.
\end{align} 
On the other hand, if $s>2$, using again the Mean Value Theorem, we obtain
\begin{align}\label{4.2.4}
\big| |\xi_{j}^{+}(a)|^{s-1}-|\xi_{j}^{+}(b)|^{s-1} \big| & \le  (s-1)|\xi_{j}^{+}(d)|^{s-2} |(\xi_{j}^{+})^{\prime}(d)| |a-b|, \quad d \in (a,b)\nonumber \\
& \le (s-1) \frac{ |(\xi_{j}^{+})^{\prime}(d)|}{D_{j-1}^{+}(d)} D_{j-1}^{+}(d) |a-b| \nonumber \\
& \le  A_{s} D_{j-1}^{+}(d) |a-b|,
\end{align}
where $A_{s}=(s-1) A_{3}$. By Lemma $\ref{lem4.2.2}$ and the Mean Value Theorem, for $y \in (-1,f_{d}(e^{-|m|+1}))$, 
we have
\begin{align}\label{4.2.5}
 |f_{d}^{j-1}\big([-1,f_{d}(e^{-|m|+1})]\big)| & = |(f_{d}^{j-1})^{\prime}(y)|[-1,f_{d}(e^{-|m|+1})]| \nonumber \\
&\ge  \frac{1}{B_{1}}D_{j-1}^{+}(d)|[-1,f_{d}(e^{-|m|+1})]|.
\end{align}           
From  inequalities \eqref{4.2.4(i)}, \eqref{4.2.4} and \eqref{4.2.5} we obtain 
\begin{align}\label{4.2.6}
\big| |\xi_{j}^{+}(a)|^{s-1}-|\xi_{j}^{+}(b)|^{s-1} \big| \le A_{s} B_{1} |a-b| \frac{|f_{d}^{j-1}\big([-1,f_{d}(e^{-|m|+1})]\big)|}{|[-1,f_{d}(e^{-|m|+1})]|},
\end{align}
Using property (A3), we have
\begin{equation}\label{4.2.7}
|[-1,f_{d}(e^{-|m|+1})]| = 1+f_{d}(e^{-|m|+1}) 
 \ge  \frac{K_0 e^{(-|m|+1)(s-1)}}{s} 
 \ge  K_0 e^{-|m|s} \ge K_0 e^{- \alpha ns},
\end{equation}
and from the binding condition, we have
\begin{align}\label{4.2.8}
|f_{d}^{j-1}\big([-1,f_{d}(e^{-|m|+1})]\big)| \le e^{-\beta j}. 
\end{align}
On the other hand, by Proposition~\ref{prop4.1.2} and the Mean Value Theorem, for some $d \in \omega$ we have
\begin{equation*}
2 \ge |\xi_{n}^{+}(\omega)| =  (\xi_{n}^{+})^{\prime}(d) |\omega| \ge  (\xi_{n}^{+})^{\prime}(d) |a-b|
 \ge  \frac{1}{A} D_{n-1}^{+}(d) |a-b|  \ge \frac{1}{A} \lambda^{n-1}|a-b|,
\end{equation*}
where the last inequality holds since $d \in R_{n-1}$. This yields   
\begin{equation}\label{byePropo}
|a-b| \le 4A \lambda^{-n}.
\end{equation}
Now, using \eqref{4.2.7}, \eqref{4.2.8} and \eqref{byePropo} in \eqref{4.2.6}, we get
\begin{align}\label{4.2.9}
\big| |\xi_{j}^{+}(a)|^{s-1}-|\xi_{j}^{+}(b)|^{s-1} \big| \le  \frac{A_{s} B_{1}}{K_0} 4A \lambda^{-n} e^{- \beta j }e^{\alpha sn}.
\end{align}
As  $e^{\alpha s}<\lambda$ for small $\alpha>0$ and  $4A \frac{A_{s} B_{1}}{K_0} (\frac{e^{\alpha s}}{\lambda})^{n} \le 1$ for  large $n$, the result follows.
\end{proof}
\begin{lemma} \label{lem4.2.5}
If $\omega \subset R_{n-1}$ is an interval such that  either $\xi_{n}^{+}(\omega)$ or $\xi_{n}^{-}(\omega)$ is contained in $I_{m}^{+}$ with $\Delta \le |m| \le [\alpha n]-1$, then there exists a constant $B_{2}=B_{2}(\alpha , \beta)>0$ such that for every $a,b \in \omega$ and every $x,y \in I_{m}^{+}$,  
$$
\frac{(f_{a}^{j})^{\prime}(f_{a}(x))}{(f_{b}^{j})^{\prime}(f_{b}(y))} \le B_{2}, \quad \forall j=1,\ldots, p(\omega,m).
$$
\end{lemma}
\begin{proof}
With no loss of generality we assume that $\xi_{n}^{+}(\omega) \subset I_{m}^{+}$. Since $x,y \in I_{m}^{+}$, then $f_{a}(x),f_{a}(y) \in [-1,f_{a}(e^{-|m|+1})]$. Thus, by Lemma $\ref{lem4.2.2}$, we have
$$
\frac{(f_{a}^{j})^{\prime}(f_{a}(x))}{(f_{b}^{j})^{\prime}(f_{b}(y))} \cdot \frac{D_{j}^{+}(a)}{D_{j}^{+}(b)} \cdot \frac{D_{j}^{+}(b)}{D_{j}(a)}
  \le  B_{1}^{2} \cdot \frac{D_{j}^{+}(a)}{D_{j}^{+}(b)}.
$$
Now, if $a=b$ then there is nothing to prove. So, let us assume that $a \neq b$. Using the chain rule, we get
$$
\frac{D_{j}^{+}(a)}{D_{j}^{+}(b)}= \frac{\prod_{i=1}^{j} f_{a}^{\prime}(\xi_{i}^{+}(a))}{\prod_{i=1}^{j} f_{b}^{\prime}(\xi_{i}^{+}(b))},
$$
which implies
\begin{align}\label{4.2.10}
\frac{D_{j}^{+}(a)}{D_{j}^{+}(b)} & =  \prod_{i=1}^{j} \Big(1+ \frac{f_{a}^{\prime}(\xi_{i}^{+}(a)) - f_{b}^{\prime}(\xi_{i}^{+}(b))}{f_{b}^{\prime}(\xi_{i}^{+}(b))}\Big) \nonumber \\
& \le  \exp\Big(\displaystyle \sum_{i=1}^{j} \Big|\frac{f_{a}^{\prime}(\xi_{i}^{+}(a)) - f_{b}^{\prime}(\xi_{i}^{+}(b))}{f_{b}^{\prime}(\xi_{i}^{+}(b))}\Big| \Big).
\end{align}
Therefore, to conclude the result we only need to prove that
$$
\displaystyle \sum_{i=1}^{j} \frac{|f_{a}^{\prime}(\xi_{i}^{+}(a)) - f_{b}^{\prime}(\xi_{i}^{+}(b))|}{f_{b}^{\prime}(\xi_{i}^{+}(b))}
$$
is uniformly bounded. Using the Mean Value Theorem, (A3) and Lemma $\ref{lem4.2.4}$, we get
\begin{align}\label{4.2.11}
f_{a}^{\prime}(\xi_{i}^{+}(a)) - f_{b}^{\prime}(\xi_{i}^{+}(b)) & \le  K_{1}|(\xi_{i}^{+}(a))|^{s-1} - K{2} |(\xi_{i}^{+}(b))|^{s-1} \nonumber \\
& \le K^{\prime} \big||(\xi_{i}^{+}(a))|^{s-1} - |(\xi_{i}^{+}(b))|^{s-1} \big|, \quad  \mbox{ fore some large $K^{\prime}$} \nonumber \\
& \le K^{\prime} e^{- \beta i}.
\end{align}
Thus, by the basic assumption and Lemma $\ref{lem4.2.4}$, we obtain
\begin{align}\label{4.2.12}
f_{b}^{\prime}(\xi_{i}^{+}(b)) & \ge  f_{a}^{\prime}(\xi_{i}^{+}(a)) - K^{\prime} e^{- \beta i} \nonumber \\
& \ge  K_{1} |\xi_{i}^{+}(a)|^{s-1} - K^{\prime} e^{- \beta i} \nonumber \\
& \ge  K_{1} e^{- \alpha (s-1)i} - K^{\prime} e^{- \beta i} \nonumber \\
& \ge  K_{1} e^{- \alpha (s-1)i}(1 - \frac{K^{\prime}}{K_{1}}e^{(\alpha(s-1) - \beta )i}) \nonumber \\
& \ge  K^{*} e^{- \alpha (s-1) i},
\end{align}
where $K^{*}=K_{1}(1- {K^{\prime}}e^{\alpha(s-1) - \beta }/{K_{1}}).$ Finally using inequalities \eqref{4.2.11},  \eqref{4.2.12} and the fact that $\beta \ge s \alpha$, we have
$$
\sum_{i=1}^{j} \frac{|f_{a}^{\prime}(\xi_{i}^{+}(a)) - f_{b}^{\prime}(\xi_{i}^{+}(b))|}{f_{b}^{\prime}(\xi_{i}^{+}(b))}  \le  \frac{K^{\prime}}{K^{*}} \sum_{i=1}^{\infty} e^{(\alpha (s-1) - \beta )i} 
 < \infty,
$$
and so the result follows.
\end{proof}
Finally, we have the following key result.
\begin{proposition}\label{lem4.2.6}
If $\omega \subset R_{n-1}$ is a parameter interval such that either $\xi_{n}^{+}(\omega)$ or $\xi_{n}^{-}(\omega)$ is contained in $I_{m}^{+}$, with $ \Delta \le |m| \le [\alpha n]-1$, then
\begin{enumerate}
\item  there exists a constant $B_{1}(\alpha , \beta)$ such that for every $k=1, \ldots, p(\omega,m)$
\begin{enumerate}
\item $ \displaystyle\frac{1}{B_{1}} \le \frac{(f_{a}^{k})^{\prime}(y)}{D_{k}^{+}(a)} \le B_{1}$, if $ y \in [-1,f_{a}(e^{-|m|+1})],$
\item$\displaystyle \frac{1}{B_{1}} \le \frac{(f_{a}^{k})^{\prime}(y)}{D_{k}^{-}(a)} \le B_{1}$, if $ y \in [f_{a}(-e^{-|m|+1}),1];$
\end{enumerate}
\item  $\displaystyle p (\omega,m)< \frac{s+1}{\beta + \log \lambda}|m|;$
\item letting $\kappa_{2}= \beta \frac{s+3}{\beta+ \log \lambda}$, for every   $a \in \omega$, $x \in I_{m}^{+}$ and $p=p(\omega,m)$ we have 
$$(f_{a}^{p+1})^{\prime}(x) \ge e^{(1 - \kappa_{2}) |m|}.$$
\end{enumerate}
\end{proposition}
\begin{proof}
By Lemma~\ref{lem4.2.2} and the definition of $p(\omega,m)$ we just need to prove item (3). We may choose $a_{*} \in \omega$ such that $p(\omega,m)=p(a_{*},m)$, then from Lemma \ref{lem4.2.5}, we have
$$   
\frac{(f_{a_{*}}^{p})^{\prime}(f_{a_{*}}(x))}{(f_{a}^{p})^{\prime}(f_{a}(x))} \le B_{2}.
$$
Now from the above inequality, using property $(A3)$, we get
\begin{align*}
\frac{|(f_{a_{*}}^{p+1})^{\prime}(x)|}{|(f_{a}^{p+1})^{\prime}(x)|} &= \frac{f_{a_{*}}^{\prime}(x)}{f_{a}^{\prime}(x)}\frac{(f_{a_{*}}^{p})^{\prime}(f_{a_{*}}(x))}{(f_{a}^{p})^{\prime}(f_{a}(x))} \\
& \le \frac{K_{1} |x|^{s-1}}{K_0 |x|^{s-1}} \frac{(f_{a_{*}}^{p})^{\prime}(f_{a_{*}}(x))}{(f_{a}^{p})^{\prime}(f_{a}(x))}
\le  \frac{K_{1}}{K_0} B_{2}.
\end{align*}
Using part $(3)$ of Lemma \ref{lem4.2.2} in the above inequality, we obtain
\begin{align*}
|(f_{a}^{p+1})^{\prime}(x)| & \ge  \frac{K_0}{K_{1} B_{3}}|(f_{a_{*}}^{p+1})^{\prime}(x)|\\
& \ge \frac{K_0}{K_{1} B_{3}}\exp \big( (1 - \beta \frac{s+2}{s+ \log \lambda})|m| \big)\\
& \ge   \exp \big( (1 - \beta \frac{s+3}{s+\log \lambda})|m| \big),
\end{align*}
where the last inequality holds provided $\Delta$ is sufficiently large.
\end{proof}

\subsection{Basic construction} \label{sec2.3}
Here we show how the sets $(R_{n})_{n \in \mathbb{N}}$ can be obtained and, for each $a \in R_{n}$, also the sequences of returns $(\gamma_{i})_{i \in \mathbb{N}}$ and bound periods $(p)_{i \in \mathbb{N}}$ as referred before. This will be obtained inductively under parameter exclusions of the initial interval $[0,a_0]$ in order to get  \eqref{(BAn)} and  (EG$_{n})$.

First we subdivide each $I_{m}$,  with $m \ge \Delta-1$, into $m^{2}$ intervals of equal length by introducing the subintervals
$$
I_{m,k}=\left[e^{-m}-k \frac{|I_{m}|}{m^{2}},e^{-m}-(k-1) \frac{|I_{m}|}{m^{2}}\right),
$$
for $1 \le k \le m^{2}$. For technical reasons, we consider also for $k\ge 1$
$$
I_{\Delta-1,k}=\left[e^{-\Delta},e^{-\Delta}+ \frac{|I_{\Delta-1}|}{(\Delta-1)^{2}}\right).
$$
%
%
%
We extend the above definitions for $m\le-(\Delta-1)$ by setting $I_{m,k}=-I_{|m|,k}$. 
Observe that each $I_{m,k}$ has two adjacent intervals: $I_{m,k-1}$ and $I_{m,k+1}$ for $I_{m,k}$ with $1<k<m^{2}$, $I_{m-1,(m-1)^{2}}$ and $I_{m,2}$ for $I_{m,1}$, and $I_{m+1,1}$ and $I_{m,m^{2}-1}$ for $I_{m,m^{2}}$. We set $I_{m,k}^{+}=I_{m_{1},k_{1}}\cup I_{m,k}\cup I_{m_{2},k_{2}}$, where $I_{m_{1},k_{1}}$ and $I_{m_{2},k_{2}}$ are the adjacent intervals to $I_{m,k}$. Note that $I_{m,k} \subset I_{m}$, $I_{m,k}^{+} \subset I_{m}^{+}$ and $|I_{m,k}^{+}| \le \frac{3|I_{m}|}{m^{2}}$ if $k\neq1$ and $|I_{m,k}^{+}| \le \frac{5|I_{m}|}{m^{2}}$ if $k=1$, provided $\Delta$ is large enough. It is also useful to consider the sets $I_{\Delta-1,(\Delta-1)^{2}}^{+}=(0,1]$ and $I_{1-\Delta,(1-\Delta)^{2}}^{+}=[-1,0)$. \par

The induction is started taking the parameter interval $[0,a_{0}]$ and the integer $N_{1}$ provided by Proposition \ref{prop4.1.5}. 
We will consider at each stage a partition $\mathcal P_{n}$ of a subset $R_{n}$ of $[0,a_0]$.
For $i=1, \cdots, N_{1}-1$, we set $R_{i}=[0,a_{0}]$ and $\mathcal{P}_{i}=\{ [0,a_{0}] \}$. We assume by induction on $n \ge N_{1}$ that the following assertions are true for every $\omega \in \mathcal{P}_{n-1}$:
\begin{enumerate}
\item There is a sequence of parameter intervals $[0,a_{0}]=\omega_{1}\supset \cdots \supset \omega_{n-1}=\omega$ such that $\omega_{k} \in \mathcal{P}_{k}$ for $k=1, \dots, n-1$.
\item There is a set $\mathscr{R}_{n-1}(\omega)=\{ \gamma_{0}, \cdots, \gamma_{\nu} \},$ with $\gamma_{0}=1$, consisting of the  \emph{return times} for $\omega$ up to $n-1$, such that   for each $k<n-1$, we have $\mathscr{R}_{k}(\omega_{k})=\mathscr{R}_{k}(\omega)\cap \{1, \cdots, k \}$. Note that when $\mathscr{R}_{n-1}(\omega)=\{1\}$, then $\omega$ has no return.
\item For each  return $\gamma_{i} \in \mathscr{R}_{n-1}(\omega)$ there are intervals $I_{m_{i},k_{i}}^{+}$ and $I_{-m_{i},k_{i}}^{+}$ with $|m_{i}| \ge \Delta$ such that $\xi_{\gamma_{i}}^{+}(\omega_{\gamma_{i}}) \subset I_{m_{i},k_{i}}^{+}$ and $\xi_{\gamma_{i}}^{-}(\omega_{\gamma_{i}}) \subset I_{-m_{i},k_{i}}^{+}$. We call $I_{m_{i},k_{i}}^{+}$ and $I_{-m_{i},k_{i}}^{+}$ the \emph{host intervals} for $\omega$ at the return $\gamma_i$. We take $p_{i}=p_{i}(\omega_{\gamma_{i}},m_{i})$, the \emph{bound period} of the return $\gamma_{i}$. For convenience we set $p_{0}=-1$. The periods 
\begin{align}\label{eq.free1}
q_{i}=\gamma_{i+1}-(\gamma_{i}+p_{i}+1) \quad \mbox{ for } i=0, \ldots, \nu-1,
\end{align}
and
\begin{align}\label{eq.free2}
q_{\nu}=\left \{
\begin{array}{l
l}      
    0 & \mbox{ if } n\le \gamma_{\nu}+p_{\nu}\\
    n-(\gamma_{\nu}+p_{\nu}+1) & \mbox{ if } n \ge \gamma_{\nu}+p_{\nu}+1.
\end{array}\right.
\end{align}
are said to be \emph{free periods} after the returns $\gamma_{i}$, for $1 \le i \le \nu$. 
\end{enumerate}
Notice that all the above properties are trivially verified for $n \le N_{1}$ taking $\mathscr{R}_{n-1}(\omega)=\{\gamma_{0}\}$. 

Now we explain how to move towards the induction step. First we consider a supplementary family~$\mathcal{Q}_{n}$ containing the portion of $\omega \in \mathcal{P}_{n-1}$ which satisfies \eqref{(BAn)}.
For each  $\omega \in \mathcal{P}_{n-1},$ there are the following possible situations:
\begin{enumerate}
\item  If $\mathscr{R}_{n-1}(\omega) \neq \{ 1 \}$ and $n \le \gamma_{\nu-1}+p_{\nu-1}$, 
then we put $\omega \in \mathcal{Q}_{n}$ and set $\mathscr{R}_{n}(\omega)=\mathscr{R}_{n-1}(\omega)$. 
\item  If either $\mathscr{R}_{n-1}(\omega) = \{ 1 \}$ or $n \le \gamma_{\nu-1}+p_{\nu-1}$ and $\xi_{n}^{\pm}(\omega) \cap U_{\Delta} \subset I_{\Delta,1} \cup I_{-\Delta,1}$, we again put $\omega \in \mathcal{Q}_{n}$ and set $\mathscr{R}_{n}(\omega)=\mathscr{R}_{n-1}(\omega)$. We call $n$ a free time for $\omega$.
\item  If we are not in the above situations, then $\omega$ must have a return situation at time~$n$. In this case we have two possibilities:
\begin{enumerate}
\item $\xi_{n}^{\pm}(\omega)$ does not cover   any interval $I_{m,k}$. 

Since $n \ge N_{1}$, we have that $\omega$ satisfies conditions $(1)$ and $(2)$ of Proposition~\ref{prop4.1.2}, and so, as mentioned before,  $\xi_{n}^{\pm}|_{\omega}$ is an isomorphism. Also, as $\omega$ is an interval by induction assumption,  $\xi_{n}^{\pm}({\omega})$ is an interval contained in some $I_{m,k}^{+}$ or $I_{-m,k}^{+}$. We put $\omega \in \mathcal{Q}_{n}$ and set $\mathscr{R}_{n}(\omega)=\mathscr{R}_{n-1}(\omega)\cup \{n\}$. We call $n$ as an inessential return time for $\omega$ and refer to $I_{m,k}^{+}$ and $I_{-m,k}^{+}$ as host intervals of the return.
\item $\xi_{n}^{\pm}(\omega)$ contains some interval $I_{m,k}$ with $|m| \ge \Delta$. 

We refer this as an essential returning situation. Consider the sets
\begin{align}
\omega_{m,k}^{\prime}&=(\xi_{n}^{+})^{-1}(I_{m,k}) \cap \omega=(\xi_{n}^{-})^{-1}(I_{-m,k}) \cap \omega,\nonumber \\
\omega^{1}&=(\xi_{n}^{+})^{-1}\big([0,1] \setminus U_{\Delta} \big) \cap \omega=(\xi_{n}^{-})^{-1}\big([-1,0] \setminus U_{\Delta} \big) \cap \omega, \label{eq.omega1}\\
\omega^{2}&=(\xi_{n}^{+})^{-1}\big([-1,0] \setminus U_{\Delta} \big) \cap \omega=(\xi_{n}^{-})^{-1}\big([0,1] \setminus U_{\Delta} \big) \cap \omega. \label{eq.omega2}
\end{align}
Letting $\mathcal{A}$ be the set of indices $(m,k)$ such that $\omega_{m,k}$ is non-empty, we have
$$
\omega \setminus (\xi_{n}^{+})^{-1}(0)=\omega \setminus (\xi_{n}^{-})^{-1}(0)=\bigcup_{(m,k) \in \mathcal{A}} \omega_{m,k}^{\prime}\cup \omega^{1} \cup \omega^{2}.
$$
Since $\xi_{n}^{\pm}|_{\omega}$ is a diffemorphism,  $\omega_{m,k}^{\prime}$ is an interval. Moreover $\xi_{n}^{+}({\omega_{m,k}^{\prime}})$ and $\xi_{n}^{-}({\omega_{-m,k}^{\prime}})$ cover completely $I_{m,k}$ and $I_{-m,k}$, respectively, except for the two extreme end intervals. We join $\omega_{m,k}^{\prime}$ to its adjacent interval if  $\xi_{n}^{+}({\omega_{m,k}^{\prime}})$ does not cover $I_{m,k}$ completely. We follow similar procedure if $\xi_{n}^{+}({\omega^{1}})$ does not cover $I_{\Delta-1,(\Delta-1)^{2}}$ or $\xi_{n}^{+}({\omega^{2}})$ does not cover $I_{1-\Delta,(1-\Delta)^{2}}$. In this way we get a new decomposition of $\omega \setminus (\xi_{n}^{+})^{-1}(0)$ into intervals $\omega_{m,k}$ such that $I_{m,k} \subset \xi_{n}^{+}(\omega_{m,k}) \subset I_{m,k}^{+}$ and $I_{-m,k} \subset \xi_{n}^{-}(\omega_{m,k}) \subset I_{-m,k}^{+}$. Now we put $\omega_{m,k} \in \mathcal{Q}_{n}$   if $m \le [\alpha n]-1$ and set $I_{m,k}^{+}$ and $I_{-m,k}^{+}$ as its host intervals. Note that the portion of $\omega$ excluded is an interval with image under $\xi_{n}^{\pm}$ contained in $U_{[\alpha n]-1}$. If $m \ge \Delta$, we set $\mathscr{R}_{n}(\omega_{m,k})=\mathscr{R}_{n-1}(\omega)\cup \{n\}$ and call $n$ an essential return for~$\omega_{m,k}$.
\end{enumerate} 
\end{enumerate}

Given $a\in\omega\in\mathcal Q_{n}$, take  $F_{n}^{\pm}(a)$ as the sum of the free periods up to time~$n$ associated to~$\xi^{\pm}$, defined as in \eqref{eq.free1} and \eqref{eq.free2}. Eventually we take
$$
\mathcal{P}_{n} = \left\{\omega \in\mathcal{Q}_{n}:F^{\pm}_{n}(a)\ge(1-\alpha)n \mbox{ for every $a\in\omega$} \right\},
$$
  and
$$
R_{n}=\bigcup_{\omega\in\mathcal{P}_{n}} \omega.
$$
Finally, we define the set  Rovella parameters as
$$
\mathcal{R}=\bigcap_{n=1}^{+\infty}R_{n}.
$$
Observe that, by construction, every $a \in R_{n}$ satisfies \eqref{(BAn)} and the \emph{free assumption}
\begin{align}\label{Fn}\tag{FA$_{n}$}
F_{n}^{\pm}(a)\ge (1-\alpha)n.
\end{align}
Using this free assumption, Rovella shows in~\cite{Ro1993} that \eqref{(EGn)} still holds for parameters in $R_{n}$, thus obtaining the exponential growth of derivative along the critical orbit. 
The strategy used by Rovella to estimate the measure of the set of parameters excluded by \eqref{Fn} is based on that used by Benedicks and Carleson in~\cite{BeCa1985, BeCa1991} for the quadratic family and uses a large deviations argument for the escape times that we introduce in the next subsection. 

We finish this subsection with a simple but useful lemma.

\begin{lemma}\label{le.escape}
If $\omega \in \mathcal{P}_{n}$, then 
$
|\omega|\le 2A\lambda^{n}.
$
\end{lemma}
\begin{proof}
By the Mean Value Theorem we have for some $a\in \omega$
$$
|\xi_{n+1}^{\pm}(\omega)| = |(\xi_{n+1}^{\pm})^{\prime}(a)|\,|\omega|.
$$
Then, by Remark~\ref{rmk4.1.3} we have
$0 \notin \xi_{n}^{+}(\omega)$, and so Proposition~\ref{prop4.1.2} gives in particular that $(\xi_{n+1}^{\pm})'(a)\neq0$. Thus,
we can write
\begin{align*}
|\omega| &= \frac{1}{|(\xi_{n+1}^{\pm})'(a)|} \left|\xi_{n+1}^{\pm}(\omega)\right| \nonumber \\
&\le \frac{1}{|(\xi_{n+1}^{\pm})'(a)|}  \nonumber \\
& =  \frac{D_{n}^{\pm}(a)}{|(\xi_{n+1}^{\pm})'(a)|} \cdot  \frac{1}{D_{n}^{\pm}(a)} .\nonumber 
\end{align*}
Now, since each $a\in \omega\in\mathcal{P}_{n}$ satisfies (EG$_n$)  and $D_{j}^{\pm}(a) \ge \eta^{j} $ for all $1\le j\le N^{\pm}$, by construction, then using Proposition~\ref{prop4.1.2}, we get the conclusion.
\end{proof}

\subsection{Escape situations} \label{escp}
Here we introduce formally the fundamental notions of escape times and escape components and deduce a key property in~Lemma~\ref{lem4.3.1} below. Take an element $\omega \in \mathcal{P}_{n-1}$ and assume that $n$ is an essential return for $\omega$.
We say that $n$ is an \emph{escape time} whenever $\xi_{n}^{\pm}({\omega})$ covers $I_{\Delta-1,(\Delta-1)^{2}}$ or $I_{1-\Delta,(1-\Delta)^{2}}$. Then, considering $\omega^{1}$ and $\omega^{2}$ as in \eqref{eq.omega1} and \eqref{eq.omega2}, we say that  $\omega^{1}$ is an \emph{escape component} in the first case, and $\omega^{2}$ an \emph{escape component} in the second case.

\begin{lemma}\label{lem4.3.1}
There is $\kappa<1$ such that if $\omega \in \mathcal{P}_{\theta}$ is an escaping component, then in the next returning situation~$\gamma$ for~$\omega$ we have
$$
|\xi_{\gamma}^{\pm}(\omega)|\ge e^{-\kappa \Delta}.
$$
\end{lemma}
\begin{proof} We consider the case $\xi^{+}$, with the other case being similar.
If $\xi_{\gamma}^{+}(\omega)$ is not completely contained in $U_{1}$, then the result follows immediately. Thus, we may assume that $\xi_{\gamma}^{+}(\omega) \subseteq U_{1}$. Since $\omega$ is an escape component with escaping time $\theta$, we have $I_{m,1}\subseteq \xi_{\theta}^{+}(\omega)$ with $|m|=\Delta-1$. With no loss of generality, assume that $m>0$. Let $p$ be the bound period after the return $\theta$ and $q=\gamma-\theta-p-1$ be the free period before the return~$\gamma$. Since $\gamma$ is the return after $\theta$,   it is not in the binding period of~$\theta$, i.e. $\gamma-\theta>p$. Now we have two possible situations:

Firstly, $\xi_{\theta}^{+}(\omega)\subseteq I_{m}$. Assuming $\omega=(a,b)$, we use Lemma \ref{lem4.1.1}, Proposition~\ref{lem4.2.6} and the Mean Value Theorem to obtain
\begin{align*}
|\xi_{\gamma}^{+}(\omega)|&=|(f_{a}^{\gamma-1}(-1),f_{b}^{\gamma-1}(-1))|=|(f_{a}^{\gamma-\theta}(f_{a}^{\theta-1}(-1)),f_{b}^{\gamma-\theta}(f_{b}^{\theta-1}(-1)))|\\
& \ge |(f_{a}^{\gamma-\theta}(f_{a}^{\theta-1}(-1)),f_{a}^{\gamma-\theta}(f_{b}^{\theta-1}(-1)))|\\
& = |f_{a}^{\gamma-\theta}(f_{a}^{\theta-1}(-1),f_{b}^{\theta-1}(-1))|\\
& = (f_{a}^{\gamma-\theta})^{\prime}(f_{c}^{\theta-1}(-1))|f_{a}^{\theta-1}(-1)-f_{b}^{\theta-1}(-1)|, \mbox{ } \text{for some $c\in \omega$.}\\
&= (f_{a}^{q})^{\prime}(f_{a}^{p+1}(f_{c}^{\theta-1}(-1)))(f_{a}^{p+1})^{\prime}(f_{c}^{\theta-1}(-1))|\xi_{\theta}^{+}(\omega)|\\
& \ge \frac{1}{e}\lambda^{q}e^{(1-\beta\frac{s+3}{\beta+\log\lambda})\Delta}|\xi_{\theta}^{+}(\omega)|, \mbox{ } \text{since $f_{c}^{\theta-1}(-1)\subset \xi_{\theta}^{+}(\omega) \subset I_{m} \subset I_{m}^{+}$.}\\
& \ge \frac{1}{e\Delta^{2}}\lambda^{q}e^{\frac{2\beta}{\beta+\log\lambda}\Delta}e^{\left(1-\beta\frac{s+5}{\beta+\log\lambda}\right)\Delta}e^{-\Delta}\\
& \ge e^{-\beta\frac{s+5}{\beta+\log\lambda}\Delta}, \mbox{ } \text{for $\Delta$ large enough,}
\end{align*}
where the second last inequality holds since $\theta$ is an escape time time for $\omega$, and so $$|\xi_{\theta}^{+}(\omega)|\ge \frac{e^{-(\Delta-1)}-e^{-\Delta}}{(\Delta-1)^{2}}>\frac{e^{-\Delta}}{\Delta^{2}}.$$ 
Secondly, $\xi_{\theta}^{+}(\omega)\supseteq I_{m}$. We have
\begin{align*}
|\xi_{\gamma}^{+}(\omega)|&=|(f_{a}^{\gamma-1}(-1),f_{b}^{\gamma-1}(-1))|=|(f_{a}^{\gamma-\theta}(f_{a}^{\theta-1}(-1)),f_{b}^{\gamma-\theta}(f_{b}^{\theta-1}(-1)))|\\
& \ge |(f_{a}^{\gamma-\theta}(f_{a}^{\theta-1}(-1)),f_{a}^{\gamma-\theta}(f_{b}^{\theta-1}(-1)))|\\
& = |f_{a}^{\gamma-\theta}(f_{a}^{\theta-1}(-1),f_{b}^{\theta-1}(-1))|\\
& \ge |f_{a}^{\gamma-\theta}(I_{m})|= (f_{a}^{\gamma-\theta})^{\prime}(x)|I_{m}|, \quad \text{for some $x\in I_{m}$.}
\end{align*}
The result follows from the above inequality similarly to the previous case. Recalling~\eqref{eq.beta} we obtain  $\kappa<1$.
\end{proof}

Let us now briefly explain how the large deviation argument is implemented by Benedicks and Carleson, giving in particular the existence of infinitely many escape times for parameters in the Rovella set $\mathcal{R}$. 
The idea is to  consider at each stage of the inductive construction the auxiliary set
\begin{align*}
R^{\prime}_{n} = \bigcup_{\omega \in \mathcal{Q}_{n}}  \omega.
\end{align*}
Given $a \in R^{\prime}_{n}$, 
take $\gamma_{1}< \cdots < \gamma_{u}$  
the return times for the parameter $a$ until time $n$, with host intervals $I_{m_{1},k_{1}}, \dots,  I_{m_{u},k_{u}}$. For convenience, we also take $\gamma_{0}=1$ and $\gamma_{u+1}=n$. Then we make a  splitting of the orbit $\{ \xi_{k}^{\pm}(a) : k=1, \dots, n-1 \}$ into periods  
\begin{align*}
P_{i}^\pm= \{ \gamma_{\ell_{i}}, \dots, \gamma_{\ell_{i+1}}-1 \}, \quad \text{$i=0, \dots, v,$}
\end{align*}
with $\ell_{0}=0$ and $\ell_{v+1}=u+1$, such that
\begin{align*}
|m_{i}| &= \Delta-1, \quad \text{for   $\ell_{2j} \le i \le  \ell_{2j+1}-1$}\\
|m_{i}| &\ge \Delta, \quad  \text{for  $\ell_{2j+1} \le i \le  \ell_{2j+2}-1$}.
\end{align*}
For the last piece we take
\begin{align*}
|m_{i}| &= \Delta-1, \quad \text{for   $\ell_{v} \le i \le  \ell_{v+1}$, if $v$ is even;}\\
|m_{i}| &\ge \Delta, \quad  \text{for   $\ell_{v} \le i \le  \ell_{v+1}$, if $v$ is odd}.
\end{align*}
Note that each  period $P^\pm_{2j}$ begins with an escape time, and all the other returns belonging to $P^\pm_{2j}$ are also escape times. Thus it consists of a piece of free orbit.
We denote by $|P^\pm_{i}|$ the number of elements in $P^\pm_{i}$ and put
\begin{align*}
T^\pm_{n}(a) = \sum_{j=0}^{v^{\prime}} |P^\pm_{2j+1}|, \quad \text{with }v^{\prime} = \left[\frac{v-1}{2}\right].
\end{align*}
We have in particular $n-T^\pm_{n}(a) \ge F^\pm_{n}(a)$. Following ideas similar to those in \cite[Subsection 2.2]{BeCa1991} (see also \cite{Mo92} for a detailed explanation), it can be obtained an estimation on the deviation of the expected value of~$T^\pm_n$, yielding 
\begin{align*}
|\{ a \in \mathcal{Q}_{n} : T^\pm_{n}(a) \ge \alpha n \}| \le e^{\ -\epsilon n } |\mathcal{R}|.
\end{align*}
This gives that the Rovella set of parameters $\mathcal R\subset [0,a_0]$ has positive measure and any $a \in \mathcal{R}$ has an infinite number of escape times.

\section{Statistical instability for the  maps} \label{sec3.3}

In this section we complete the proof of Theorem~\ref{thmain}. We start by extracting from assumptions (A0)-(A6) in Subsection~\ref{sec1.1.3} some useful  facts about the map $f_{a}$, for $a\in[0,a_{0}]$ with $a_{0}$ sufficiently close to~$0$.
Recall that each $f_{a}$ is differentiable in $I \setminus \{ 0 \}$, with $f_{a}^{\prime \prime}(x)<0$ for $x \in [-1,0)$ and $f_{a}^{\prime \prime}(x)>0$ for $x \in (0,1]$. 
Moreover,  $\pm 1$ are the critical values for $f_{a}$ with $f_{a}(-1)$ close to $-1$ and $f_{a}(1)$ close to $1$. Hence, the graph of $f_{a}$ has two connected components. 
\begin{figure}[ht]
\centering
\includegraphics[width=0.4\textwidth]{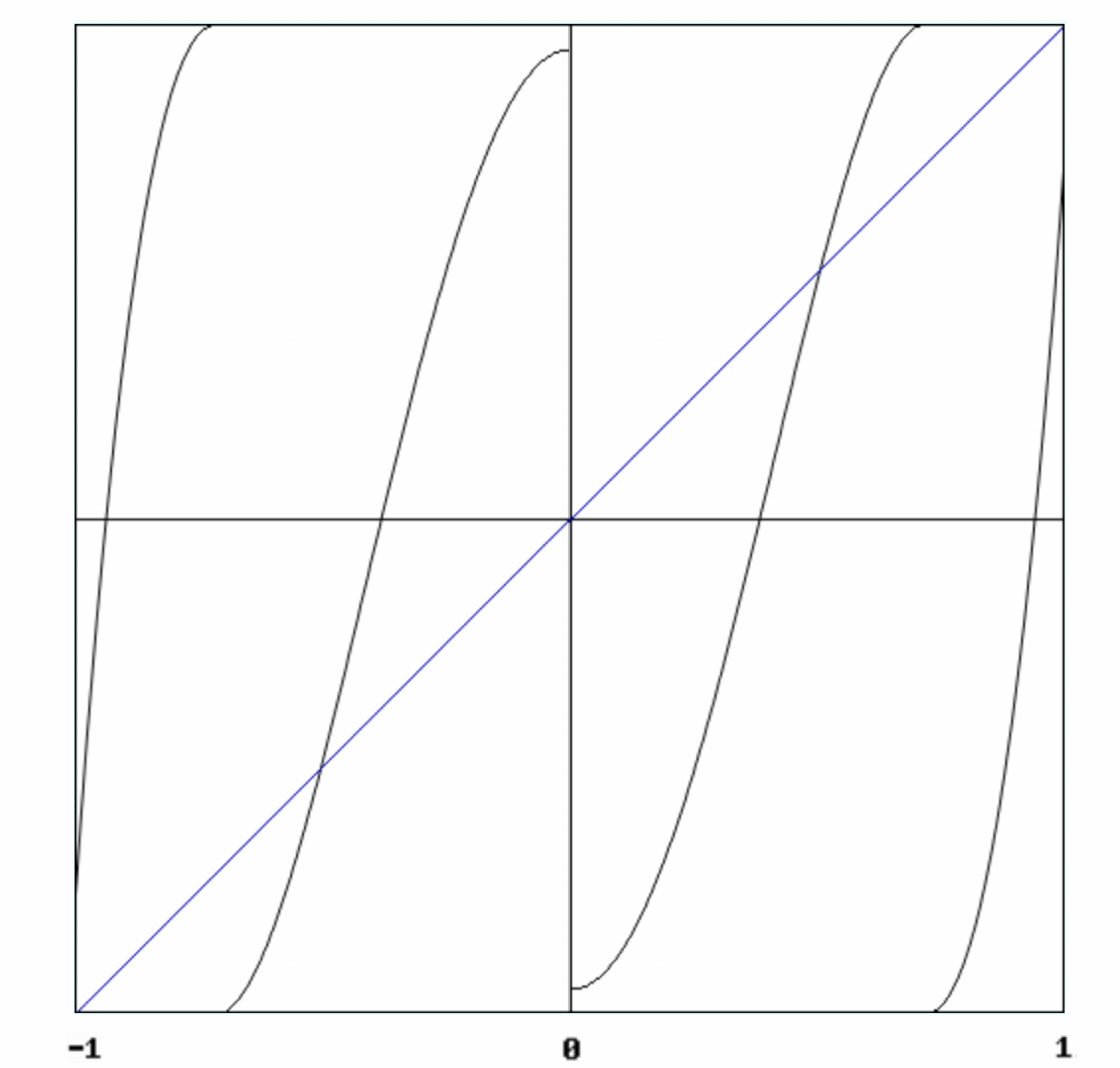}
\caption{Graph of $f_{a}^{2}$}\label{sec}
\end{figure}
This further suggests that the graph of~$f_{a}^{2}$ consists of four connected components, 
corresponding to the intervals $\left[-1,O^{-}_a\right)$, $\left(O^{-}_a,0\right)$, $\left(0,O^{+}_a\right)$ and $\left(O^{+}_a,1\right]$, 
where $O^{-}_a$ and $O^{+}_a$ are the zeros of~$f_{a}$ located on the left and the right side of $0$, respectively; see Figure~\ref{sec}. 
%
For each $a\in   [0,a_0]$, consider $\{ y^{-}_{a}, y^{+}_{a} \}$ the period two repelling orbit for~$f_{a}$, with $y_{a}^{-}<0$ and $y_{a}^{+}>0$. 
  
\begin{proposition} \label{mainlem} If $a_{0}$ is sufficiently close to 0 and $\Delta$ is sufficiently large, 
then for each escape time $\theta$ with escape component $\omega\in\mathcal P_\theta$ and $\gamma$ the next returning situation for $\omega$ we can find a parameter $a \in \omega\cap \mathcal R$ and an integer $\ell \ge 1$ such that $f_{a}^{\gamma+\ell}(-1)=y_{a}^{-}$.
\end{proposition}

\begin{proof}
Since $\gamma$ is a returning time for $\omega \in\mathcal P_\theta$, we have $\xi_{\gamma}^{+} (\omega ) \cap (-\delta, \delta) \neq \emptyset$. Moreover, as~$\gamma$ is the first return after the escape time $\theta$,    Lemma \ref{lem4.3.1}  gives   
$ 
|\xi_{\gamma}^{+} (\omega) |  \ge \delta^{\kappa} .
$ 
%
Without loss of generality, we may assume that the interval $\xi_{\gamma }^{+} (\omega)$ lies on the right hand side of zero, and  so there are $b,c\in\omega$ such that 
$$
\xi_{\gamma}^{+}(b) = \delta \quad\text{and}\quad \xi_{\gamma}^{+}(c) = \delta^{\kappa}\gg\delta.
$$
Using (A3) and the Mean Value Theorem, we get
 \begin{align} \label{5.1.2}
\left| -1-\xi_{\gamma+1}^{+}(b)\right|=\left|f_{b}(0)-f_{b}(\delta)\right| \le K_{1}\delta^{s}
\end{align}
and
\begin{align} \label{5.1.3}
\left| -1-\xi_{\gamma+1}^{+}(c)\right|=\left|f_{c}(0)-f_{c}(\delta^{\kappa})\right|\ge K_{0}\delta^{\kappa s}.
\end{align}
Taking $\Delta\in \mathbb N$  sufficiently large, such that for $\delta=e^{-\Delta}$ we have
$$
K_{1}\delta^{(1-\kappa)s}< \frac{K_{0}}2,
$$
and using~\eqref{5.1.2} and \eqref{5.1.3}, 
we obtain 
\begin{equation}\label{eq.length}
\left| \xi_{\gamma+1}^{+}(b)-\xi_{\gamma+1}^{+}(c)\right|\ge K_{0}\delta^{\kappa s}-K_{1}\delta^{s}=\left(K_{0}-K_{1}\delta^{(1-\kappa) s}\right)\delta^{\kappa s}\ge \frac{K_{0}\delta^{\kappa s}}2.
\end{equation}
On the other hand, from the assumptions in Subsection~\ref{sec1.1.3} we easily deduce the existence of $x_{0}\in (-1,0)$ and $M >1$ such that, for $a_{0}$ sufficiently close to~$0$, we have for all $a\in[0,a_{0}]$ and all $x\in [0,x_{0}]$ 
\begin{equation}\label{eq.bounds}
f_{a}^{\prime}(x)\ge M.
\end{equation}
Now consider   the sequence of pre-images $\cdots<y_{0}^{j}<\cdots <y_{0}^{1}<y_{0}^{0}=y_{0}^{-}$, with $f_{0}(y_{0}^{j})=y_{0}^{j-1}$ for all $  j\ge 1$. Take $j_1$ the first   integer for which
\begin{equation}\label{eq.yj1}
\left|-1-y_{0}^{j_1}\right|<K_{1}\delta^{s}.
\end{equation}
Considering $0\le j_0<j_1$ the first integer such that $y_{0}^{j_0}<x_{0}$, we further require that 
\begin{equation}\label{eq.j1j0}
\left(\frac{1}{M}\right)^{j_1-j_0-1}<\frac{K_{0}\delta^{\kappa s}}6.
\end{equation}
Then, using \eqref{eq.bounds} and~\eqref{eq.j1j0}, we easily deduce that for each $j\ge j_1$
\begin{equation}\label{eq.fundamental}
\left|y_{0}^{j}-y_{0}^{j-1}\right|\le\left(\frac{1}{M}\right)^{j-j_0-1}\left|y_{0}^{j_0+1}-y_{0}^{j_0}\right|\le\left(\frac{1}{M}\right)^{j-j_0-1}<\frac{K_{0}\delta^{\kappa s}}6.
\end{equation}
Taking $\ell=j_1+1$, it follows from \eqref{eq.length} and~\eqref{eq.fundamental} that   the points $y_{0}^{\ell-1} , y_{0}^{\ell} $ and $y_{0}^{\ell+1}$  belong to the interval $ \xi_{\gamma+1}^{+}(\omega)$. This ensures the existence of an interval centred at $y_0^\ell$ of size 
 $$\min\left\{y_{0}^{\ell} - y_{0}^{\ell+1},  y_{0}^{\ell-1} - y_{0}^{\ell} \right\}.$$
Clearly, this size does not depend on $\omega$, $\theta$ or $\gamma$. Then, using Proposition~\ref{prop4.1.2}, we ensure that $ \xi_{\gamma+1+\ell}^{+}(\omega)$ contains an open interval $J$ centred at $y_0^-$  not depending on $\omega$, $\theta$ or~$\gamma$.
Now, since $\{ y^{-}_{0}, y^{+}_{0} \}$ is a hyperbolic set for $f_0$, it has a hyperbolic continuation $\{ y^{-}_{a}, y^{+}_{a} \}$  depending continuously on the parameter~$a$. 
Thus, taking $a_0$ sufficiently close to 0,  we   still assure that $y_{a}^-\in J\subset  \xi_{\gamma+1+\ell}^{+}(\omega)$ for every $a\in [0,a_0]$. 
Since $\omega\subset [0,a_0]$, it follows  that the continuous function   $y_{a}^--\xi_{\gamma+1+\ell}^{+}(a)$ must have a zero in the interval~$\omega$. This means that  $f_{a}^{\gamma+\ell}(-1)= y_{a}^-$ for some $a\in\omega$. Finally, as the orbit of the critical value $-1$ falls onto a repelling periodic orbit under iterations by $f_a$,   the parameter $a$ necessarily belongs to the set~$\mathcal R$.
\end{proof}

\begin{proposition} \label{pr.super}
Assume that $a\in\mathcal{R}$ and  $\xi^{+}_{k}(a)\in   \{y_{a}^{-},y_{a}^{+}\}$ for some $k\ge1$. Then there is a sequence of parameters $(a_{n})_{n}$ converging to $a$ such that each $f_{a_{n}}$ has a super-attractor and
$$\mu_{a_{n}}\stackrel{w^{*}}{\longrightarrow} \frac{1}{2}\left(\delta_{y^{-}_{a}}+\delta_{ y^{+}_{a}}\right),\quad\text{as $n\to\infty$,}
$$
where $\mu_{a_{n}}$ denotes the probability measure supported on the super-attractor of $f_{a_{n}}$. 
\end{proposition}

\begin{proof}
Let $L\ge 0$ be the smallest positive integer such that $\xi^{+}_{L}(a)\in \{y_{a}^{-},y_{a}^{+}\}$. Assume for definiteness that $\xi^{+}_{L}(a)=y_{a}^{-}$. Fixing $r>0$ small, we define the intervals 
$$
Y_{r}^{-}=(y^{-}_{a}-r, y^{-}_{a}+r)\quad\text{and}\quad Y_{r}^{+}=(y^{+}_{a}-r, y^{+}_{a}+r).
$$
Using Lemma~\ref{lem4.1.1} and Proposition~\ref{prop4.1.2} 
we can find a sequence of parameter intervals $(\Omega_{n})_{n}$ with 
\begin{equation}\label{eq.Omega}
\{a\}=\bigcap_{n\ge1} \Omega_n,
\end{equation}
 and a  sequence of positive integers $L=m_{1}<m_{2}<\cdots$ such that for every $n\ge 1$ we have
\begin{itemize}
\item [(a)]  $\Omega_{n+1}\subset\Omega_{n}$;
\item [(b)] $\xi_{i}^{+}(\Omega_{n})\subset Y_{r}^{-}\cup Y_{r}^{+}$ for all $L\le i\le m_{n}$;
\item [(c)] $\xi_{m_{n}}^{+}(\Omega_{n})=Y_{r}^{-}$ or $\xi_{m_{n}}^{+}(\Omega_{n})=Y_{r}^{+}$;
\item [(d)]   $m_{n}-L=2t_{n}-1$ for some integer $t_{n}$.
\end{itemize}
It follows from (c) that there are $N=N(r)>0$ and $\rho_{n}\in\mathbb N$ with $\rho_{n} \le N$ such that 
\begin{equation}\label{eq.ro}
0\in\xi_{m_{n}+\rho_{n}}^{+}(\Omega_{n}),\quad\text{for all $n\ge 1$}.
\end{equation}
As a consequence of \eqref{eq.Omega} and \eqref{eq.ro}, we obtain a sequence $a_{n}\in\Omega_{n}$  with $a_{n}\to a$ as $n\to \infty$, such that $f_{a_{n}}$ has a super-attractor of period $m_{n}+\rho_{n}$ for every $n\ge1$. 
Now take any continuous $\varphi: I \rightarrow \mathbb{R}$ and fix $\varepsilon>0$ sufficiently small. 
For each $n\ge1$ we have
\begin{align} \label{5.2.2}
\int\varphi \, d \mu_{a_{n}}&= \frac{1}{m_{n}+\rho_{n}}\sum\limits_{i=1}^{m_{n}+\rho_{n}}\varphi\left(f_{a_{n}}^{i}(-1)\right)  \nonumber \\
& \le  \frac{1}{m_{n}+1}\sum\limits_{i=L}^{m_{n}}\varphi\left(f_{a_{n}}^{i}(-1)\right)+ \frac{L+N}{m_{n}+1}\|\varphi\|_0,
\end{align}
where $\|\varphi\|_{0}$ stands for the $C^{0}$-norm of $\varphi$. Since the second term in the inequality above clearly goes to zero as $ {n}\rightarrow \infty$, we are going to work out the first term. By the uniform continuity of $\varphi$ on the closed interval $I$, we can choose $r>0$ small such that 
\begin{align} \label{5.2.3}
|\varphi(x)-\varphi(y)|< {\varepsilon} , \quad\mbox{ whenever  } |x-y| < r.
\end{align}
On the other hand, since we are assuming $f_{a }^{L}(-1)=y_{a}^{-}$, it follows from (b) that
\begin{align} \label{5.2.4}
|f_{a_{n}}^{i}(-1) - f_{a}^{i-L} (y_{a}^{-}))| < r,  \quad \mbox{for all $L\le i \le m_{n}$}.
\end{align}
Then, using \eqref{5.2.3} and \eqref{5.2.4}, we obtain
\begin{align} \label{5.2.5}
 \sum\limits_{i=L}^{m_{n}}\varphi\left(f_{a_{n}}^{i}(-1)\right) & =   \sum\limits_{i=L}^{m_{n}}\left(\varphi\left(f_{a }^{i-L} (y_{a}^{-})\right)+\varphi\left(f_{a_{n}}^{i}(-1)\right)-\varphi\left(f_{a }^{i-L} (y_{a}^{-})\right)\right) \nonumber \\
& \le  \sum\limits_{i=L}^{m_{n}}\left(\varphi\left(f_{a}^{i-L} (y_a^-)\right)+  {\varepsilon} \right)\nonumber\\
&= \sum\limits_{i=1}^{m_{n}-L+1} \left(\varphi\left(f_{a}^{i-1} (y_{a}^{-})\right)+  {\varepsilon} \right) .
\end{align}
Recalling that from (d) we can write $m_{n}-L+1=2t_{n}$ for some positive integer $t_{n}$, we get
\begin{align} \label{5.2.6}
\sum\limits_{i=1}^{m_{n}-L+1}  \varphi\left(f_{a}^{i-1} (y^{-}_{a}) \right)
& =  \sum\limits_{i=1}^{2t_{n} }  \varphi\left(f_{a}^{i-1} (y^{-}_{a})\right)     \nonumber \\
& =t_{n}\left( \varphi(y_{a}^{-})+\varphi(y_{a}^{+}) \right)  \nonumber \\
& = \frac{m_{n}-L+1}2 \left( \varphi(y_{a}^{-})+\varphi(y_{a}^{+}) \right).
\end{align}
Using \eqref{5.2.2}, \eqref{5.2.5} and \eqref{5.2.6}, we obtain
$$
\int\varphi \mbox{ }  d \mu_{a_{n}} \le  \frac{m_{n}-L+1}{m_{n}+1} \left(\frac12\left( \varphi(y_{a}^{-})+\varphi(y_{a}^{+})\right) + {\varepsilon} \right)+ \frac{ L+N}{m_{n}+1}\|\varphi\|_0.
$$  
Similarly, we get
$$
\int\varphi \mbox{ } d \mu_{a_{n}} \ge  \frac{m_{n}-L+1}{m_{n}+N} \left(\frac12\left( \varphi(y_{a}^{-})+\varphi(y_{a}^{+})\right) - {\varepsilon} \right)- \frac{ L+N}{m_{n}+1}\|\varphi\|_0.
$$  
Using that $m_{n}\rightarrow \infty$ when $n\to\infty$ and, since $\varepsilon>0$ is arbitrary, we have for each $\varphi:I\to\mathbb{R}$ continuous
$$
 \lim_{n\to\infty}\int\varphi \mbox{ } d \mu_{a_{n}}=  \frac12\left( \varphi(y_{a}^{-})+\varphi(y_{a}^{+})\right),
$$ 
which clearly gives the desired conclusion.
\end{proof}

Let us now finish the proof of Theorem~\ref{thmain}.
Given any $a\in \mathcal{R}$, from Lemma~\ref{le.escape} and Proposition~\ref{mainlem} we obtain a sequence $(a_{n})_{n}$ in $ \mathcal{R}$ converging to $a$ for which the orbit of $-1$ under $f_{a_{n}}$ is pre-periodic to $\{ y^{-}_{a_{n}}, y^{+}_{a_{n}} \}$. 
Since $a_n\in\mathcal{R}$, 
by Proposition~\ref{pr.super} we obtain for each $n\in\mathbb N$ a sequence $(a_{n,k})_{k}$ converging to $ a_{n}$ when $k\to \infty$, such that $f_{a_{n,k}}$ has a super-attractor 
  and 
  \begin{equation}\label{eq.supermeasure}\mu_{a_{n,k}}\stackrel{w^{*}}{\longrightarrow} \frac{1}{2}\left(\delta_{y^{-}_{a_{n}}}+\delta_{ y^{+}_{a_{n}}}\right),\quad\text{as $k\to\infty$,}
\end{equation}
where $\mu_{a_{n,k}}$ is the probability measure supported on a super-attractor of~$f_{a_{n,k}}$. 
%
%
%

Now observe that as any  $f_{a}$   is smooth on the intervals $[-1,0)$ and $(0,1]$,   we may find a neighbourhood $\mathcal{N}$ of the hyperbolic set $\{ y^{-}_{a}, y^{+}_{a} \}$ such that $f_{a}$ is smooth on $\mathcal{N}$. Therefore,  the set $\{ y^{-}_{a}, y^{+}_{a} \}$ varies continuously with the parameter~$a\in\mathcal R$; see e.g. \cite{DeSt2012}.
Together with \eqref{eq.supermeasure}, this enables us to obtain a sequence $(a_{n,k_{n}})_{n}$ with $a_{n,k_{n}} \rightarrow a$ as $n\rightarrow \infty$ such that  
$$
\mu_{a_{n,k_{n}}}\stackrel{w^{*}}{\longrightarrow} \frac{1}{2}\left(\delta_{y^{-}_{a}}+\delta_{ y^{+}_{a}}\right),\quad\text{as $n\to\infty$.}
$$
Since $\mu_{a_{n,k_{n}}}$ is the probability measure supported on a super-attractor of~$f_{a_{n,k_{n}}}$, we have proved Theorem~\ref{thmain}.


\begin{thebibliography}{10}

\bibitem{Al2004}
J.~F. Alves.
\newblock Strong statistical stability of non-uniformly expanding maps.
\newblock {\em Nonlinearity}, 17(4):1193--1215, 2004.

\bibitem{AlBoVi2000}
J.~F. Alves, C. Bonatti and M. Viana.
\newblock {SRB} measures for partially hyperbolic systems whose central
direction is mostly expanding.
\newblock {\em  Invent. Math.}, 140(2):351--398, 2000.

\bibitem{ACF1}  J. F. Alves, M. Carvalho and J. M. Freitas.
\newblock Statistical stability for H\'enon maps of the Benedicks-Carleson type. 
\newblock {\em Ann. Inst. H. Poincar\'e Anal. Non Lin\'eaire}, 27(2):595--637, 2010.

\bibitem{ACF2}  J. F. Alves, M. Carvalho and J. M. Freitas.
\newblock Statistical stability and continuity of SRB entropy for systems with Gibbs-Markov structures. 
\newblock {\em Comm. Math. Phys.}, 296(3):73--767, 2010.

\bibitem{AlSo2012}
J.~F. Alves and M. Soufi.
\newblock Statistical stability and limit laws for Rovella maps.
\newblock {\em Nonlinearity}, 25:3527--3552, 2012.

\bibitem{AlSo2019}
J.~F. Alves and M. Soufi.
\newblock Statistical stability for Rovella flows.
\newblock \emph{In preparation}.

\bibitem{AlSo2014}
J.~F. Alves and M. Soufi.
\newblock Statistical stability of geometric Lorenz attractors.
\newblock {\em  Fund. Math.}, 224:219--231, 2014.

\bibitem{AlVi2002}
J.~F. Alves and M. Viana.
\newblock Statistical stability for robust classes of maps with non-uniform
expansion.
\newblock {\em Ergodic Theory Dynam. Systems}, 22(1):1--32, 2002.


\bibitem{ArPa2010}
V. Ara{\'u}jo, V. and  M.~J. Pacifico.
\newblock Three-dimensional flows, volume 53 of Ergebnisse der Mathematik und
ihrer Grenzgebiete. 3. folge. A series of Modern surveys in Mathematics. Springer, Heidelberg, 2010.

\bibitem{BR18} W. Bahsoun  and M. Ruziboev, 
\newblock On the statistical stability of Lorenz attractors with a $C^{1+\alpha}$ stable foliation, 
\newblock {\em Ergodic Theory Dynam. Systems}, to appear.

\bibitem{BS09}  V. Baladi and D. Smania,
\newblock Analyticity of the SRB measure for holomorphic families of quadratic-like Collet-Eckmann maps. 
\newblock {\em Proc. Amer. Math. Soc.} 137(4):1431--1437, 2009. 

\bibitem{BS12}  V. Baladi and D. Smania,
\newblock Linear response for smooth deformations of generic nonuniformly hyperbolic unimodal maps. \newblock {\em Ann. Sci. \'Ec. Norm. Sup\'er. (4)} 45(6):861--926, 2012. 

\bibitem{BeCa1985}
M. Benedicks and L. Carleson.
\newblock On iterations of $1-ax^{2}$ on $(-1, 1)$.
\newblock {\em Ann. of Math.}, 122(1):1--25, 1985.

\bibitem{BeCa1991}
M.~Benedicks and L.~Carleson.
\newblock The dynamics of the H{\'e}non map.
\newblock {\em Ann. of Math.}, 133(1):73--169, 1991.

\bibitem{BeYo1992}
M.~Benedicks and L.-S. Young.
\newblock Absolutely continuous invariant measures and random perturbations for
certain one-dimensional maps.
\newblock {\em Ergodic Theory Dynam. Systems}, 12(1):13--37, 1992.


\bibitem{BoRu1975}
R.~Bowen and D.~Ruelle.
\newblock The ergodic theory of {A}xiom {A} flows.
\newblock {\em Invent. Math.}, 29(3):181--202, 1975.


\bibitem{DeSt2012}
W. De~Melo and S. Van~Strien.
\newblock {\em One-dimensional dynamics}, volume~25.
\newblock Springer Science \& Business Media, 2012.

\bibitem{Fr2005}
J.~M. Freitas.
\newblock Continuity of {SRB} measure and entropy for {B}enedicks-{C}arleson
quadratic maps.
\newblock {\em Nonlinearity}, 18(2):831--854, 2005.

%
\bibitem{Fr2010}
J.~M. Freitas.
\newblock Exponential decay of hyperbolic times for {B}enedicks--{C}arleson
quadratic maps.
\newblock {\em Port. Math}, 67:525--540, 2010.

\bibitem{Gu1979}
J.~Guckenheimer  (1979).
\newblock Sensitive dependence to initial conditions for one dimensional maps.
\newblock {\em Comm. Math. Phys.}, 70(2):133--160.

\bibitem{GuWi1979}
J.~Guckenheimer and R.~F. Williams.
\newblock Structural stability of {L}orenz attractors.
\newblock {\em Inst. Hautes \'Etudes Sci. Publ. Math.}, (50):59--72, 1979.

\bibitem{HoKe1990}
F.~Hofbauer and G.~Keller.
\newblock Quadratic maps without asymptotic measure.
\newblock {\em Comm. Math. Phys.}, 127(2):319--337, 1990.

\bibitem{Ja1981}
M.~V. Jakobson.
\newblock Absolutely continuous invariant measures for one-parameter families
of one-dimensional maps.
\newblock {\em Comm. Math. Phys.}, 81(1):39--88, 1981.

\bibitem{Ke1982}
G.~Keller.
\newblock Stochastic stability in some chaotic dynamical systems.
\newblock {\em Monatsh. Math.}, 94(4):313--333, 1982.

\bibitem{Lo1963}
E.~N. Lorenz.
\newblock Deterministic non-periodic flow.
\newblock {\em J. Atmos. Sci.}, 20:130--141, 1963.


\bibitem{Me2000}
R.~J. Metzger.
\newblock Sinai-{R}uelle-{B}owen measures for contracting {L}orenz maps and
flows.
\newblock {\em Ann. Inst. H. Poincar{\'e} Anal. Non Lin{\'e}aire},
17(2):247--276, 2000.

\bibitem{Me-Stochastic2000}
R.~J. Metzger.
\newblock Stochastic stability for contracting {L}orenz maps and
flows.
\newblock {\em Comm. Math. Phys.},
212(2):277--296, 2000.



\bibitem{Mo92}
F. J. S. Moreira.
\newblock Chaotic dynamics of quadratic maps.
MSc thesis, University of Porto, 1992.
https://cmup.fc.up.pt/cmup/fsmoreir/downloads/BC.pdf

\bibitem{Ro1993}
A.~Rovella.
\newblock The dynamics of perturbations of the contracting {L}orenz attractor.
\newblock {\em Bol. Soc. Brasil. Mat. (N.S.)}, 24(2):233--259, 1993.

\bibitem{Ru1976}
D.~Ruelle.
\newblock A measure associated with axiom-{A} attractors.
\newblock {\em Amer. J. Math.}, 98(3):619--654, 1976.

\bibitem{RS92}
M.~Rychlik and E.~Sorets.
\newblock Regularity and other properties of absolutely continuous invariant
measures for the quadratic family.
\newblock {\em Comm. Math. Phys.}, 150(2):217--236, 1992.

\bibitem{Si1972}
J.~G. Sinai.
\newblock Gibbs measures in ergodic theory.
\newblock {\em Uspehi Mat. Nauk}, 27(4(166)):21--64, 1972.

\bibitem{Sm1998}
S. Smale.
\newblock Mathematical problems for the next century.
\newblock {\em Math. Intelligencer}, 20(2):7--15, 1998.


\bibitem{T96} M. Tsujii.
\newblock On continuity of Bowen-Ruelle-Sinai measures in families of one-dimensional maps. 
\newblock{\em Commun. Math. Phys.}, 177(1): 1--11, 1996.

\bibitem{Th2001}
H.~Thunberg.
\newblock Unfolding of chaotic unimodal maps and the parameter dependence of
natural measures.
\newblock {\em Nonlinearity}, 14(2):323--337, 2001.


\bibitem{Tu1999}
W.~Tucker.
\newblock The {L}orenz attractor exists.
\newblock {\em C. R. Acad. Sci. Paris S\'er. I Math.}, 328(12):1197--1202,
1999.

\bibitem{Ur1995}
R. Ures.
\newblock On the approximation of {H}énon-like attractors by homoclinic tangencies. 
\newblock {\em Ergodic Theory Dynam. Systems}, 15(6): 1223--1229, 1995. 

\bibitem{Ur1996}
R. Ures.
\newblock Hénon attractors: SBR measures and Dirac measures for sinks. 
\newblock {\em International Conference on Dynamical Systems, (Montevideo, 1995)}, 214--219, Pitman Res. Notes Math. Ser., 362, Longman, Harlow, 1996.

\bibitem{Vi1997}
M. Viana.
\newblock {\em Stochastic dynamics of deterministic systems}, Brazilian Mathematics Colloquium.
\newblock IMPA, Rio de Janeiro, 1997.

  

\end{thebibliography}

\end{document}